\documentclass[11pt]{article}

\usepackage[latin1]{inputenc}
\usepackage{amsthm,amsmath,amssymb,bm}
\usepackage{graphicx}
\usepackage{color}
\usepackage{verbatim}
\usepackage{float}
\usepackage{mathrsfs}
\usepackage{hyperref}

\usepackage{array}   
\newcolumntype{C}{>{$}c<{$}}
\newcolumntype{L}{>{$}l<{$}}

\usepackage[short labels]{enumitem}
\setlist[enumerate]{topsep=0pt,itemsep=-1ex,partopsep=1ex,parsep=1ex}
\setlist[itemize]{topsep=0pt,itemsep=-1ex,partopsep=1ex,parsep=1ex}

\usepackage{bbm}

\theoremstyle{plain}
\newtheorem{theo}{Theorem}[section]

\newtheorem{lemma}[theo]{Lemma}

\newtheorem{conj}[theo]{Conjecture}

\newtheorem{claim}[theo]{Claim}

\theoremstyle{definition}
\newtheorem{defn}[theo]{Definition}

\newtheorem{que}[theo]{Question}
\newtheorem{obs}[theo]{Observation}

\newcommand{\mc}[1]{\mathcal{#1}}
\newcommand{\mb}[1]{\mathbb{#1}}

\newcommand{\ms}[1]{\mathscr{#1}}

\newcommand{\sm}{\setminus}
\newcommand{\ov}{\overline}

\newcommand{\wt}{\widetilde}
\newcommand{\eps}{\varepsilon}
\newcommand{\es}{\emptyset}

\newcommand{\wh}{\widehat}

\newcommand{\aA}{\alpha}
\newcommand{\bB}{\beta}
\newcommand{\gG}{\gamma}
\newcommand{\dD}{\delta}

\newcommand{\kK}{\kappa}

\newcommand{\lL}{\lambda}

\newcommand{\sS}{\sigma}

\newcommand{\sd}{\bigtriangleup}

\newcommand{\DD}{\Delta}

\newcommand{\ind}{\mathbbm{1}}

\newcommand{\Cols}{\ms{C}}

\newcommand{\abs}{\mathrm{abs}}

\newcommand{\col}{\mathrm{col}}

\newcommand*\ECone{\raisebox{0pt}{\includegraphics{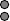}}}
\newcommand*\ECtwo{\raisebox{0pt}{\includegraphics{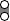}}}

\usepackage[dvipsnames]{xcolor}

\title{Stability of transversal Hamilton cycles and paths}
\author{Yangyang Cheng\thanks{Mathematical Institute,
University of Oxford, Oxford, UK. Email: yangyang.cheng@maths.ox.ac.uk.}
\thanks{Yangyang Cheng was supported by a PhD studentship of ERC Advanced Grant 883810.}
\and Katherine Staden\thanks{School of Mathematics and Statistics, The Open University, Walton Hall, Milton Keynes, UK. Email: katherine.staden@open.ac.uk.}
\thanks{Katherine Staden was supported by EPSRC Fellowship EP/V025953/1.}}

\begin{document}

\maketitle

\begin{abstract}
Given graphs $G_1,\ldots,G_s$ all on a common vertex set and a graph $H$
with $e(H) = s$,
a copy of $H$ is \emph{transversal} or \emph{rainbow} if it contains
one edge from each $G_i$.
We establish a stability result for transversal Hamilton cycles: the minimum degree required to guarantee a transversal Hamilton cycle can be lowered as long as the graph collection $G_1,\ldots,G_n$ is far in edit distance from several extremal cases.
We obtain an analogous result for Hamilton paths.
The proof is a combination of our newly developed regularity-blow-up method for transversals, along with the absorption method.
\end{abstract}

\section{Introduction}

Given graphs $G_1,\ldots,G_s$ on a common vertex set and a graph $H$ with $e(H) = s$, a \emph{transversal} copy of $H$ is a copy of $H$ containing exactly one edge from each of the graphs $G_1,\ldots,G_s$. We often think of each $G_i$ having the colour $i$, and thus we also call a transversal copy of $H$ a \emph{rainbow} copy. The following general question was formulated by Joos and Kim in \cite{JK}.
\begin{que}\label{que}
Let $H$ be a graph with $s$ edges and let $\bm{G}=(G_1,\ldots,G_s)$ be a collection of 
graphs on a common vertex set $V$. 
Which properties imposed on $\bm{G}$ yield a transversal copy of $H$?
\end{que}
Note that when $G_1=\ldots=G_s=G\in \bm{G}$, a transversal copy of $H$ inside $\bm{G}$ is equivalent to a subgraph $H$ of $G$. So the question of Joos and Kim is a generalisation of the classical embedding problem.
An important line of research for the classical problem has been to seek sufficient conditions for a graph to contain a copy of $H$,
which is of particular interest as the problem is often computationally difficult.
Question~\ref{que} is the transversal generalisation.

\subsection{General transversal embedding}
There are transversal analogues of several classical embedding theorems.
Mantel's well-known theorem from 1907 states that any $n$-vertex graph with more than $\lfloor n^2/4 \rfloor$ edges contains a triangle.
Aharoni, DeVos, de la Maza, Montejano and {\v{S}}{\'a}mal~\cite{Aharoni} showed that 
in any collection of three graphs on the same vertex set, each one must have significantly more
than $n^2/4$ edges to guarantee a transversal triangle.
It is still a major open problem in this area to generalise their result to arbitrary $K_r$.

Dirac's theorem on Hamilton cycles from 1952 states that in an $n$-vertex graph $G$,
minimum degree $\dD(G) \geq n/2$ suffices to guarantee a Hamilton cycle.
Aharoni~\cite{Aharoni} conjectured that the transversal generalisation holds:
given graphs $G_1,\ldots,G_n$ on a common vertex set of size $n$, if $\dD(G_i)\geq n/2$ for each $i\in [n]$, then there exists a transversal Hamilton cycle. This conjecture was asymptotically verified by Cheng, Wang and Zhao \cite{Cheng1} and later fully proved by Joos and Kim \cite{JK}, who used a short `elementary' argument.

As in Dirac's theorem, to guarantee a transversal copy of a spanning graph,
it is natural to impose a minimum degree condition on every graph in a collection.
Given a graph $H$ and a positive integer $n$, the \emph{transversal minimum degree threshold} of $H$ is the minimum $d$ such that if every graph in a collection of $e(H)$ $n$-vertex graphs has minimum degree at least $d$, then the collection contains a transversal copy of $H$.
If the transversal minimum degree threshold is, up to an additive $o(n)$ term, the same as the minimum degree threshold (for a single $n$-vertex graph)
then we say that $H$ is \emph{colour-blind}.
Note that the transversal minimum degree threshold
is always at least as large as the minimum degree threshold
for embedding in a single graph, 
which can be seen by taking a graph collection of identical copies of one graph.

Joos and Kim~\cite{JK} showed that perfect matchings are colour-blind (in fact without any error term).
Given a graph $F$ with $v(F) | n$, an $n$-vertex $F$-factor $H$ is a spanning graph consisting of vertex-disjoint copies of $F$. 
Cheng, Han, Wang and Wang~\cite{Cheng2}
and independently, Montgomery, M\"uyesser and Pehova~\cite{MMP},
showed that 
the $K_k$-factor is colour-blind.
The authors of~\cite{MMP} further
characterised those $F$ for which $F$ is colour-blind,
and determined the correct transversal minimum degree threshold when $F$ is not colour-blind.
These results are approximate transversal versions of the Hajnal-Szemer\'edi theorem~\cite{HS}
and a theorem of K\"uhn and Osthus~\cite{KOfactor}.
In~\cite{MMP} it was also shown that spanning trees with maximum degree $o(n/\log(n))$
are colour-blind,
generalising a well-known
result of Koml\'os, S\'ark\"ozy and Szemer\'edi~\cite{KSStree}.

Gupta, Hamann, M\"uyesser, Parczyk and Sgueglia~\cite{GHMPS} showed that the $(k-1)$-th power $C^{k-1}_n$ of a Hamilton cycle is colour-blind,
as are certain hypergraph cycles. 
Chakraborti, Im, Kim and Liu~\cite{CIKL} generalised many of the results in this section by proving a `transversal bandwidth theorem' which gives an upper bound on the transversal minimum degree threshold for a large class of graphs, which is tight in many cases, and hence gives a rich class of graphs $H$ which are colour-blind.

Other than the results of Joos and Kim on Hamilton cycles and perfect matchings~\cite{JK},
there are few results which determine the exact transversal minimum degree threshold.
It would be interesting to improve all of the results in this section in this direction
and determine whether a given $H$ is `exactly colour-blind'.
We make further remarks on this at the end of the paper in Section~\ref{sec:conclude}.

\subsection{Transversal embedding of Hamilton cycles}
Following Joos and Kim's transversal generalisation of Dirac's theorem, there has been much progress on further questions related to the Hamiltonian case of Question~\ref{que} for graph collections on $n$ vertices. 
Cheng, Wang and Zhao~\cite{Cheng1} also showed that minimum degree $(n+1)/2$ is sufficient to guarantee a rainbow cycle of every length $3,\ldots,n-1$.
Bradshaw~\cite{bradshaw} proved analogues of the results of Joos and Kim in bipartite graph collections.
Chakraborti, Kim, Lee and Seo~\cite{CKLS} studied the problem of finding transversal Hamilton cycles in tournaments (a tournament is a complete oriented graph).
Li, Li and Li~\cite{LLL} and Zhang and van Dam~\cite{ZVD} imposed Ore-type conditions rather than a minimum degree condition to find transversal Hamilton paths and cycles.
Bradshaw, Halasz and Stacho~\cite{BHS} showed that minimum degree at least $n/2$ guarantees not just one transversal Hamilton cycles as per~\cite{JK}, but exponentially many.

Transversal Hamilton cycles have also been well studied in hypergraphs.
Extending a result of R\"odl, Ruci\'nski and Szemer\'edi~\cite{RSS1,RSS},
Cheng, Han, Wang, Wang and Yang~\cite{TWWY} proved an asymptotically optimal bound
on the minimum codegree guaranteeing a tight Hamilton cycle in a uniform hypergraph.
The main result in~\cite{GHMPS} implies asymptotically optimal results for other types of Hamilton cycles in uniform hypergraphs, for various degree conditions.
Answering a question left open here, the notable case of vertex degree for tight $3$-uniform Hamilton cycles was obtained by Tang, Wang, Wang and Yan~\cite{TWWY}.

The problem of finding a Hamilton cycle with a given colouring in a graph collection (rather than a transversal Hamilton cycle which is \emph{some} rainbow cycle) was studied by Bowtell, Morris, Pehova and Staden in~\cite{BMPS}.
It turns out that minimum degree $(\frac{1}{2}+o(1))n$ is enough to guarantee such a cycle,
but at least $\lfloor n/2\rfloor+1$ is required.
Ferber, Han and Mao~\cite{FHM} and
Anastos and Chakraborti~\cite{AC} studied \emph{robustness} of transversal Hamiltonicity.
In~\cite{AC} the threshold at which a graph collection consisting of random subgraphs of graphs satisfying Dirac's condition contains a transversal Hamilton cycle is determined.
This generalises a result of Krivelevich, Lee and Sudakov~\cite{KLS}.

\subsection{Stability for transversal Hamilton cycles and paths}\label{sec:stabintro}

In this paper, we are interested in the stability phenomenon for Hamilton paths and cycles,
and will show that graph collections which do not contain a Hamilton cycle, but whose graphs all have minimum degree slightly smaller than that needed to guarantee one, in fact have a special structure.

We define two extremal graphs $\ECone$ and $\ECtwo$ of order $n$ as follows:
\begin{itemize}
\item the $n$-vertex graph $\ECone$ is the union of two disjoint cliques of size as equal as possible;
\item the $n$-vertex graph $\ECtwo$ is a complete bipartite graph whose two parts have size as equal as possible. 
\end{itemize}
Since Dirac's theorem shows that a Hamilton cycle is guaranteed in any $n$-vertex graph $G$ with minimum degree $\dD(G) \geq n/2$, graphs satisfying this minimum degree condition are called \emph{Dirac graphs}.
As mentioned in the previous section, Joos and Kim~\cite{JK} showed that any collection of $n$ Dirac graphs contains a transversal Hamilton cycle. 

Dirac's theorem is tight:
the graph $\ECone$ does not contain a Hamilton cycle, and when $n$ is even it has minimum degree $n/2-1=\lceil n/2\rceil-1$.
The graph $\ECtwo$ does not contain a Hamilton cycle when $n$ is odd, in which case it has minimum degree $(n-1)/2=\lceil n/2\rceil-1$ and contains no Hamilton cycle.
In fact, when $n$ is odd, we can add arbitrary edges into the smaller part of $\ECtwo$
without creating a Hamilton cycle. 
For any $n$, we can perturb either of these examples slightly, at the expense of slightly reducing the minimum degree, to get a non-Hamiltonian graph.

In fact, it is well-known that these are the only examples.
Given $\kK>0$, we say that two graphs $G,G'$ on the same vertex set of size $n$ are \emph{$\kK$-close}
if one can add and remove at most $\kK n^2$ edges of $G$ to obtain $G'$.

\begin{theo}[Stability for Dirac's theorem, folklore]\label{th:dirstab}
For all $\kK>0$, there exist $\mu>0$ and $n_0$ such that the following holds for all integers $n \geq n_0$. Let $G$ be an $n$-vertex graph with $\dD(G) \geq (\frac{1}{2}-\mu)n$. If $G$ contains no Hamilton cycle, then either $G$ is $\kK$-close to $\ECtwo$; or $G$ is $\kK$-close to $\ECone$ with arbitrary edges added inside one part. 
\end{theo}

Given a graph collection $\bm{G}$, we write $\dD(\bm{G}) := \min_{G \in \bm{G}}\dD(G)$ for its \emph{minimum degree}.
Now suppose that a graph collection $\bm{G}=(G_1,\ldots,G_n)$ on a common vertex set of size $n$
with $\dD(\bm{G}) \geq (\frac{1}{2}-\mu)n$
has no transversal Hamilton cycle
(we say $\bm{G}$ is \emph{non-Hamiltonian}).
Of course, one can take $G_i=G$ where $G$ is non-Hamiltonian;
that is, we can take $\bm{G}$ to be a collection of identical copies of one of the graphs in Theorem~\ref{th:dirstab}.
On the other hand, taking copies of $\ECone$ and $\ECtwo$ with, say, uniformly random partitions
will give a graph collection which \emph{is} Hamiltonian with high probability.
However, one can combine these graphs to get further collections without a transversal Hamilton cycle.

\begin{defn}[$\bm{H}^b_a$, half-split graph collection] 
Given integers $a,b\geq 0$, let $\bm{H}_a^b$ be the graph collection on a common vertex set of size $n$ obtained by taking $a$ copies of $\ECone$ and $b$ copies of $\ECtwo$ where they are defined on the same equitable partition $A\cup B$.

We say that a graph collection $\bm{J}$ 
on a common vertex set $V$ of size $n$ is \emph{half-split} if there is $A \subseteq V$ with $|A|=\lfloor n/2\rfloor+1$
such that $J[A]=\es$ and $J[A,V \sm A]$ is complete for all $J \in \bm{J}$.
\end{defn}

 Note that $\dD(\bm{H}_a^b)=\lfloor n/2\rfloor-\ind(a > 0)$.
Note that $\bm{H}_n^0$ is non-Hamiltonian since $\ECone$ is, and when $n$ is odd, $\bm{H}^n_0$ is also non-Hamiltonian since $\ECtwo$ is.
We claim that when $b$ is odd and $a+b=n$, the collection $\bm{H}_a^b$ is non-Hamiltonian. Suppose not, and let $C$ be a transversal Hamilton cycle and cyclically direct its edges. We say an edge of $C$ is of $1$-type if it comes from an $\ECone$ copy and of $2$-type if it comes from an $\ECtwo$ copy. Note that the number of $2$-type edges directed from $A$ to $B$ is equal to the number of $2$-type edges directed from $B$ to $A$. Hence, the total number of $2$-type edges in $C$ is even, which implies that $b$ is even, a contradiction.

A half-split graph collection $\bm{J}$ on $n$ vertices
has $\dD(\bm{J})=\lfloor n/2\rfloor-1$.
Further, $\bm{J}$ is non-Hamiltonian
since if it contained any Hamilton cycle $C$ (with any colours), then between each pair of vertices in $A$ on $C$ there is at least one vertex of $V \sm A$, so $|V \sm A| \geq |A|$, a contradiction.

Given $\kK>0$, we say that two graph collections $\bm{G},\bm{G'}$ on the same common vertex set of size $n$ are \emph{$\kK$-close} if one can add and remove at most $\kK n^3$ edges of graphs in $\bm{G}$ to obtain $\bm{G'}$.
Our first result shows that any non-Hamiltonian collection of $n$ almost-Dirac graphs on $n$ vertices
is close to one of these two families.

\begin{theo}\label{th:rhc}
For all $\kK>0$, there exist $\mu>0$ and $n_0$ such that the following holds for all integers $n \geq n_0$.
Let $\bm{G}=(G_1,\ldots,G_n)$ be a graph collection on a common vertex set $V$ of size $n$, with $\dD(\bm{G})\geq (\frac{1}{2}-\mu)n$. If $\bm{G}$ contains no transversal Hamilton cycle, then either
$\bm{G}$ is $\kK$-close to $\bm{H}_a^b$ for some $a\in [n]$ where $b=n-a$ is odd; or
$\bm{G}$ is $\kK$-close to a half-split graph collection.
\end{theo}

In contrast to the Hamilton cycle case, our second result demonstrates 
that half-split graph collections and $\bm{H}_{n-1}^0$ are the only extremal constructions (in an approximate sense) for transversal Hamilton paths.
Clearly $\bm{H}^0_{n-1}$ does not contain a transversal Hamilton path,
while if we increase the size of $A$ by one when $n$ is odd, a half-split collection of $n-1$ graphs does not contain a transversal Hamilton path either.

\begin{theo}\label{th:rhp}
For all $\kK>0$, there exist $\mu>0$ and $n_0$ such that the following holds for all integers $n \geq n_0$. Let $\bm{G}=(G_1,\ldots,G_{n-1})$ be a graph collection on a common vertex set $V$ of size $n$, with $\dD(\bm{G})\geq (\frac{1}{2}-\mu)n$. If $\bm{G}$ does not contain a transversal Hamilton path, then $\bm{G}$ is either $\kK$-close to $\bm{H}_{n-1}^0$ or $\bm{G}$ is $\kK$-close to a half-split graph collection.
\end{theo}

We prove these theorems in a unified manner, combining our newly-developed regularity-blow-up method from~\cite{Cheng4} and the absorption method for transversals,
which uses ideas from the papers~\cite{Cheng2,Cheng3,Cheng1} of the first author.
We give a sketch of the proofs in Section~\ref{sec:sketch}.

\medskip
\noindent
\textbf{Remark.}
In the final stages of preparation of this manuscript, we learned that in the work of Anastos and Chakraborti on robust Hamiltonicity~\cite{AC} mentioned earlier, an auxiliary result on the stability of transversal Hamilton cycles is proved (Theorem~4.4 in~\cite{AC}).
In terms of the focus of this paper, their result is slightly weaker than Theorem~\ref{th:rhc} and states that any graph collection $\bm{G}$ satisfying the conditions of Theorem~\ref{th:rhc} is close to $\bm{H}^b_0 \cup \bm{J}_a$ where $\bm{J}_a$ is a half-split collection of $a=n-b$ graphs, and the common independent set $A$ in $\bm{J}_a$ is one of the parts of the graphs in $H$.
This was sufficient for their main result mentioned earlier on transversal Hamilton cycles in random subgraphs of Dirac graphs.
In turn, their result is slightly stronger than the `stable case' of Theorem~\ref{th:rhc}. This is the part of our proof which uses the regularity-blow-up method and absorption method.
Anastos and Chakraborti instead use iterative absorption in their proof.
Additionally, their result is stronger in another direction: they show that when a collection is far from $\bm{H}^b_0 \cup \bm{J}_a$, not only is there a Hamilton cycle,
there is a certain measure on the set of Hamilton cycles
which in particular implies there are many of them, improving the result of~\cite{BHS}.

\subsection{Notation and organisation}

\textbf{Notation.}
Let $G$ be any graph. We denote its vertex set by $V(G)$ and its edge set by $E(G)$. We write $v(G)=|V(G)|$ and $e(G)=|E(G)|$ for their sizes. Given $v\in V(G)$, the \emph{neighbourhood} $N_G(v)$ of $v$ is the set of vertices that are incident to $v$ and the \emph{degree} of $v$ is $d_G(v):=|N_G(v)|$.
For any $U\subseteq V(G)$, let $G[U]$ be the induced graph of $G$ on $U$, i.e., graph with vertex set $U$ and those edges of $G$ with both endpoints in $U$. Let $G-U:=G[V(G)\sm U]$
and $G-v := G-\{v\}$. 
For each vertex $v\in V(G)$ and subset $U \subseteq V(G)$, let $N_G(v,U)=N_G(v)\cap U$ and $d_G(v,U)=|N_G(v,U)|$. 
Let $E_{G}(X,Y)$ be the set of edges with one endpoint in $X$ and another in $Y$ (so edges in $G[X \cap Y]$ are only counted once) and let $e_G(X,Y)=|E_G(X,Y)|$.
We write $E_G(X) := E_G(X,X) = E(G[X])$.
We write $P=v_1\ldots v_t$ to denote a path of \emph{length $t$},
and will sometimes write $v_1Pv_t$ or $v_1P$ or $Pv_t$ for $P$
if we wish to emphase its endpoint(s).

Given any collection $\bm{G}=(G_c: c \in \Cols)$ of graphs on a common vertex set $V$,
we call $\Cols$ the \emph{colour set} of $\bm{G}$. 
For any two sets $X,Y \subseteq V$, let $E_{\bm{G}}(X,Y)$ be the multiset of edges of any colour with one endpoint in $X$ and another in $Y$ and $e_{\bm{G}}(X,Y)=|E_{\bm{G}}(X,Y)|=\sum_{G \in \bm{G}}e_G(X,Y)$. 
Given any edge-coloured graph $H$ with $V(H) \subseteq V$ with edge-colouring $\sS : E(H) \to \Cols$,
we write $\col(H) := \bigcup_{e \in E(H)}\sS(e)$ for the set of colours used on $H$.

We say a constant $x=a\pm b$ if we have $a-b\leq x\leq a+b$. For any two constants $\aA,\bB \in (0,1)$, we write $\aA \ll \bB$ if there exists function $\aA_0=\aA_0(\bB)$ such that the subsequent arguments hold for all $0<\aA\leq \aA_0$. When we write multiple constants in a hierarchy, we mean that they are chosen from right to left. For any two integers $a \leq b$, let $[a,b]=\{a\leq x\leq b: x\in \mathbb{Z}\}$
and $[a] := [1,a]$.
Given a set $X$ and positive integer $k$, we write
\begin{itemize}
\item $\binom{X}{k}$ for the set of all $k$-subsets of $X$,
\item $X^k$ for the set of all $k$-tuples of elements of $X$,
\item $(X)_k$ for the set of all $k$-tuples of \emph{distinct} elements of $X$.
\end{itemize}
We use script letters e.g.~$\ms{C},\ms{A}$ to denote sets of colours
and bold letters e.g.~$\bm{G},\bm{J}$ to denote graph collections.

\medskip
\noindent
\textbf{Organisation.}
We conclude this section with a sketch of the proofs of Theorems~\ref{th:rhc} and~\ref{th:rhp}.
In Section~\ref{sec:reg}, we introduce our regularity-blow-up method for transversals
and state the definitions and tools we will need later.
The remainder of the paper contains the proofs of our two main results.
We prove them in a unified manner and all of the auxiliary results along the way
concern the setting of Theorem~\ref{th:rhc}, i.e.~collections of $n$ graphs and transversal Hamilton cycles (as opposed to paths).
Section~\ref{sec:extstab} introduces suitable notions of extremality and stability for graphs and graph collections; the proof will be split into a `stable case' and an `extremal case'.
In Section~\ref{sec:absorb}, we prove some results about absorption
that will be used in the proof of the stable case.
Then, in Section~\ref{sec:stable} we show that there is always a transversal Hamilton cycle in this case.
Section~\ref{sec:ext} deals with the extremal case.
We combine the results of Sections~\ref{sec:extstab}--\ref{sec:ext} to give a unified proof of Theorems~\ref{th:rhc} and~\ref{th:rhp} in Section~\ref{sec:proof}.
In Section~\ref{sec:conclude} with finish with some concluding remarks
and a discussion of `exact' results for transversal embedding.

\subsection{Probabilistic tools}

We use the following version of Chernoff's bound:
\begin{lemma}\label{lm:chernoff}
Let $X$ be a random variable with binomial or hypergeometric distribution,
and let $0<\eps<\frac{3}{2}$.
Then
$$
\mathbb{P}[|X-\mb{E}(X)|\geq \eps \mb{E}(X)]\leq 2e^{-\frac{\eps^2}{3}\mb{E}(X)}.
$$
\end{lemma}


\subsection{Sketch of the proof of Theorems~\ref{th:rhc} and~\ref{th:rhp}}\label{sec:sketch}

We will define two kinds of stability for a graph family $\bm{G}=(G_1,\ldots,G_n)$ with minimum degree at least $(\frac{1}{2}-o(1))n$ (see Definition~\ref{def:stab}).
We will say that
\begin{itemize}
\item $\bm{G}$ is \emph{strongly stable} if $\bm{G}$ contains many graphs which are not close to containing either $\ECone$ or $\ECtwo$, while 
\item $\bm{G}$ is \emph{weakly stable} if almost all graphs in $\bm{G}$ are close to containing either $\ECone$ or $\ECtwo$, but the vertex partitions associated with these subgraphs are not similar. 
\end{itemize}
We say that $\bm{G}$ is \emph{stable} if it is either strongly stable or weakly stable.
Thus if $\bm{G}$ is not stable, most graphs in $\bm{G}$ are close to containing either $\ECone$ or $\ECtwo$ and their associated vertex partitions are almost the same.  The first part of the proof is to show that if $\bm{G}$ is stable, then $\bm{G}$ contains a transversal Hamilton cycle.
We then deal with the extremal case, where $\bm{G}$ is not stable. 

\medskip
\noindent
{\bf Step 1}. \emph{Build the absorbing cycle when $\bm{G}$ is stable.} (Section~\ref{sec:absorb})

\noindent
We build an `absorbing cycle' $C$ for $\bm{G}$
with the property that $C$ is very small and
there is a set $\mc{A}$ consisting of \emph{some} colour-vertex pairs $(c,v)$ and colour-vertex-vertex triples $(c,u,v)$
such that whenever $\mc{A}_0 \subseteq \mc{A}$ is sufficiently small compared to $|C|$,
 $C$ can \emph{absorb} all of its elements (see Definition~\ref{def:absorb}).
This property implies
\begin{itemize}
\item we can add any $(c,v) \in \mc{A}_0$ to the cycle $C$,
which means there is a new cycle $C'$ with 
colour set $\col(C')=\col(C) \cup \{c\}$
and vertex set $V(C')=V(C) \cup \{v\}$;
and, 
\item for any $(c,u,v) \in \mc{A}_0$, whenever $P$ is a transversal path with endpoints $u,v$, we can add $P$ into $C$ using only the new colour $c$.
\end{itemize}
In fact, when $\bm{G}$ is strongly stable, $\mc{A}$ contains \emph{all} pairs and triples. 
But in general, we need to construct an additional auxiliary set to absorb those colours and vertices which $C$ cannot.
After $C$ is constructed, we delete $C$ and its colour set from $\bm{G}$.

\noindent
{\bf Step 2}. \emph{Use the regularity-blow-up method for transversals to cover with long paths.}

\noindent
We apply the regularity lemma for graph collections (Lemma~\ref{lm:weakregcol}) to $\bm{G}$ and thus get a reduced graph collection $\bm{R}$ that inherits the minimum degree condition of $\bm{G}$. By randomly partitioning its colour set, we obtain two families $\bm{R}^1$ and $\bm{R}^2$, each using about half of the colour clusters. Since $\bm{R}^1$ and $\bm{R}^2$ inherit the original degree condition, we are able to find two almost perfect transversal matchings $M_1$, $M_2$ from $\bm{R}^1$ and $\bm{R}^2$ (see Lemma~\ref{lm:2matching}).
The union of these matchings is `locally' like a Hamilton cycle.
We apply the transversal blow-up lemma (Theorem~\ref{th:blowup}) to obtain almost spanning disjoint transversal paths inside each edge of $M_1 \cup M_2$, which cover almost all the vertices outside $C$. 

\noindent
{\bf Step 3}. \emph{Connect the paths and cover remaining vertices via the absorbing cycle.}

\noindent
The last step is to use the absorbing property of $C$ as well as the auxiliary set to connect all the paths to a transversal Hamilton cycle.

\noindent
{\bf Step 4}. \emph{The extremal case.} (Section~\ref{sec:ext})

\noindent
The remaining case is when $\bm{G}$ is not stable and thus most graphs in $\bm{G}$ are close to containing either $\ECone$ or $\ECtwo$ and their vertex partitions are almost the same partition $V=A \cup B$.
To prove Theorem~\ref{th:rhc}, we need to show that if $\bm{G}$ is not close to a half-split graph collection or close to $\bm{H}_a^b$ for some $a\in [n]$ where $b=n-a$ is odd, then $\bm{G}$ contains a transversal Hamilton cycle. 
For this, we first find a short path that covers atypical vertices and colours, and in some cases, balances the two sides of the partition.
Next, we find partitions $A=A^0 \cup A^{10} \cup A^{11}$
and $B=B^0 \cup B^{10} \cup B^{11}$
whose sizes depend on the exact number of graphs that are close to $\ECone$ and to $\ECtwo$.
The transversal blow-up lemma guarantees that there are spanning transversal paths in each of $\bm{G}[A^0,B^0]$, $\bm{G}[A^{10},A^{11}]$, $\bm{G}[B^{01},B^{11}]$ with disjoint colour sets
and endpoints in any given large subsets.
The structure of $\bm{G}$ allows us to find such paths which can be concatenated into
a transversal Hamilton cycle, proving Theorem~\ref{th:rhc}.
The proof of Theorem~\ref{th:rhp} differs only at the end:
we use the fact that given a collection $\bm{G}$ of $n-1$ graphs on a common vertex set of size $n$,
the $n$-graph collection obtained from adding a complete graph to $\bm{G}$ has a Hamilton cycle
if and only if $\bm{G}$ has a Hamilton path.

\section{The regularity-blow-up method for transversals}\label{sec:reg}

In this section, we introduce the tools developed in our paper~\cite{Cheng4}.
We first define (super)regularity for graph collections.

\begin{defn}[Regularity and superregularity]
Suppose that $\bm{G}$ is a graph collection with colour set $\Cols$, where each $G_c$ is bipartite with parts $V_1,V_2$. We say that
\begin{itemize}
\item $\bm{G}$ is \emph{$(\eps,d)$-regular} if
whenever $V_i' \subseteq V_i$ with $|V_i'| \geq \eps|V_i|$ for $i=1,2$
and $\Cols' \subseteq \Cols$ with $|\Cols'| \geq \eps|\Cols|$,
we have 
$$
\left| \frac{\sum_{c \in \Cols'}e_{G_c}(V_1',V_2')}{|\Cols'||V_1'||V_2'|}
-\frac{\sum_{c \in \Cols}e_{G_c}(V_1,V_2)}{|\Cols||V_1||V_2|} \right| < \eps
$$
and $e_{\bm{G}}(V_1,V_2) \geq d|\Cols||V_1||V_2|$.
\item $\bm{G}$ is \emph{$(\eps,d)$-superregular} if it is $(\eps,d)$-regular and 
$\sum_{c \in \Cols}d_{G_c}(x) \geq d|\Cols||V_{3-i}|$ for all $x \in V_i$ where $i\in[2]$, and
$e(G_c) \geq d|V_1||V_2|$ for all $c \in \Cols$.
\end{itemize}
\end{defn}

Note that if every $G_c$ with $c \in \Cols$ is the same, then $\bm{G}$ is $(\eps,d)$-regular if and only if $G_c$ is $(\eps,d)$-regular; and $\bm{G}$ is $(\eps,d)$-superregular if and only if $G_c$ is $(\eps,d)$-superregular.

\begin{lemma}[Typical vertices and colours \cite{Cheng4}]\label{lm:standard}
Let $0<\eps \ll d \leq 1$, and let $\bm{G}$ be an $(\eps,d)$-regular graph collection with colour set $\Cols$, where each $G_c$ is bipartite with parts $V_1,V_2$. Then the following hold:
\begin{itemize}
\item[(i)] for every $i \in [2]$ and all but at most $\eps|V_i|$ vertices $v \in V_i$ we have $\sum_{c \in \Cols}d_{G_c}(v) \geq (d-\eps)|V_{3-i}||\Cols|$;
\item[(ii)] for all but at most $\eps|\Cols|$ colours $c \in \Cols$ we have $e(G_c) \geq (d-\eps)|V_1||V_2|$.
\end{itemize}
\end{lemma}

\begin{lemma}[Slicing lemma \cite{Cheng4}]\label{lm:slice}
Let $0<1/n \ll \eps \ll \aA \ll d \leq 1$, and let $\bm{G}$ be a graph collection with colour set $\Cols$, where each $G_c$ is bipartite with parts $V_1,V_2$ each of size at least $n$.
Suppose that $\bm{G}$ is $(\eps,d)$-regular.
Let $V_i' \subseteq V_i$ with $|V_i'| \geq \aA|V_i|$ for $i \in [2]$ and $\Cols' \subseteq \Cols$ with $|\Cols'| \geq \aA|\Cols|$. Then $\bm{G}':=(G_c[V_1',V_2']:c \in \Cols')$
is $(\eps/\aA,d/2)$-regular.
\end{lemma}

We use the following regularity lemma for graph collections.

\begin{lemma}[Regularity lemma for graph collections \cite{Cheng4}]\label{lm:weakregcol}
For all integers $L_0 \geq 1$ and every $\eps,\dD>0$, there is an $n_0=n_0(\eps,\dD,L_0)$ such that
for every $d \in [0,1)$ and
every graph collection $\bm{G}=(G_c: c \in \Cols)$ on vertex set $V$ of size $n \geq n_0$ with $\dD n \leq |\Cols| \leq n/\dD$, there exists a partition of $V$ into $V_0,V_1,\ldots,V_L$, of $\Cols$ into $\Cols_0,\Cols_1,\ldots,\Cols_M$ and a spanning subgraph $G'_c$ of $G_c$ for each $c \in \Cols$ such
that the following properties hold:
\begin{enumerate}[(i)]
\item $L_0 \leq L,M \leq n_0$ and $|V_0|+|\Cols_0| \leq \eps n$;
\item $|V_1|=\ldots=|V_L|=|\Cols_1|=\ldots = |\Cols_M| =: m$;
\item $\sum_{c \in \Cols}d_{G'_c}(v) > \sum_{c \in \Cols}d_{G_c}(v)-(3d/\dD^2+\eps)n^2$ for all $v \in V$ and
$e(G'_c) > e(G_c)-(3d/\dD^2+\eps)n^2$ for all $c \in \Cols$;
\item if, for $c \in \Cols$, the graph $G'_c$ has an edge with both vertices in a single cluster $V_i$ for some $i \in[L]$, then $c \in \Cols_0$;
\item for all triples $(\{h,i\},j) \in \binom{[L]}{2} \times [M]$, we have that either $G'_c[V_h,V_i]=\es$ for all $c \in \Cols_j$, or
$\bm{G}'_{hi,j} := (G'_c[V_h,V_i]: c \in \Cols_j)$ is $(\eps,d)$-regular.
\end{enumerate}
\end{lemma}

The sets $V_i$ are called \emph{vertex clusters} and the sets $\Cols_j$ are called \emph{colour clusters}, while $V_0$ and $\Cols_0$ are the \emph{exceptional} vertex and colour sets respectively.

\begin{defn}[Reduced graph collection]
Given a graph collection $\bm{G}=(G_c: c \in \Cols)$ on $V$ and parameters $\eps>0, d \in [0,1)$ and $L_0 \geq 1$, the \emph{reduced graph collection} $\bm{R}=\bm{R}(\eps,d,L_0)$ of $\bm{G}$ is defined as follows.
Apply Lemma~\ref{lm:weakregcol} to $\bm{G}$ with parameters $\eps,\dD,d,L_0$
to obtain $\bm{G}'$ and a partition $V_0,\ldots,V_L$ of $V$ and $\Cols_0,\ldots,\Cols_M$ of $\Cols$ where
$V_0$, $\Cols_0$ are the exceptional sets and $V_1,\ldots,V_L$ are the vertex clusters
and $\Cols_1,\ldots,\Cols_M$ are the colour clusters.
Then
$\bm{R}=(R_1,\ldots,R_M)$ is a graph collection of $M$ graphs each on the same vertex set $[L]$,
where, for $(\{h,i\},j) \in \binom{[L]}{2} \times [M]$, we have $hi \in R_j$ whenever $\bm{G}'_{hi,j}$ is $(\eps,d)$-regular.
\end{defn}

The next lemma
states that clusters inherit a minimum degree bound in the reduced graph from $G$.

\begin{lemma}[Degree inheritance \cite{Cheng4}]\label{lm:inherit}
Suppose $L_0 \geq 1$ and $0 < 1/n \ll \eps \leq d \ll \dD,\gG,p \leq 1$.
Let $\bm{G}=(G_c: c \in \Cols)$ be a graph collection on a vertex set $V$ of size $n$ with $\dD(G_c) \geq (p+\gG)n$ for all $c \in \Cols$ and $\dD n \leq |\Cols| \leq n/\dD$.
Let $\bm{R}=\bm{R}(\eps,d,L_0)$ be the reduced graph collection of $\bm{G}$ on $L$ vertices with $M$ graphs.
Then
\begin{enumerate}[(i)]
\item for every $i \in [L]$ there are at least $(1-d^{1/4})M$ colours $j \in [M]$ for which $d_{R_j}(i) \geq (p+\gG/2)L$;
\item for every $j \in [M]$ there are at least $(1-d^{1/4})L$ vertices $i \in [L]$ for which $d_{R_j}(i) \geq (p+\gG/2)L$.
\end{enumerate}
\end{lemma}

All of the above definitions and lemmas are merely convenient restatements 
of `weak regularity' for $3$-uniform hypergraphs, specialised to the transversal setting.
However, the next lemma, our transversal blow-up lemma, which was the main result of~\cite{Cheng4},
is a non-trivial tool which we will use in the proofs of Theorems~\ref{th:rhc}
and~\ref{th:rhp}.

An $n$-vertex graph $H$ is \emph{$\mu$-separable} if there is $X \subseteq V(H)$ of size at most $\mu n$ such that $H-X$ consists of components of size at most $\mu n$.
This class of graphs includes many natural graphs including $F$-factors, $2$-regular graphs and powers of Hamilton cycles. 
In this paper, we will only use it for Hamilton paths.
Thus it suffices to state a simplified version of the main result of~\cite{Cheng4}.

\begin{theo}[Transversal blow-up lemma \cite{Cheng4}]\label{th:blowup}
Let $0 < 1/n \ll \eps,\mu,\aA, \ll \nu,d,\dD,1/\DD \leq 1$.
Let $\Cols$ be a set of at least $\dD n$ colours and let $\bm{G} = (G_c: c \in \Cols)$ be a collection of bipartite graphs with the same vertex partition $V_1,V_2$,
where $n \leq |V_1|\leq|V_2|\leq n/\dD$, such that
$\bm{G}$ is $(\eps,d)$-superregular.
Let $H$ be a $\mu$-separable bipartite graph with parts $A_1,A_2$ of sizes $|V_1|,|V_2|$ respectively, and $|\Cols|$ edges and maximum degree $\DD$.
Suppose further that, for $i=1,2$, there is a set $U_i \subseteq A_i$ with $|U_i| \leq \aA|A_i|$
and for each $x \in U_i$, a \emph{target} set $T_x \subseteq V_i$ with $|T_x| \geq \nu|V_i|$.
Then $G$ contains a transversal copy of $H$ such that for $i=1,2$, every $x \in U_i$
is embedded inside $T_x$.
\end{theo}


\section{Extremality and stability}\label{sec:extstab}

\subsection{Extremal and stable graphs}

In this section, we define extremality for a single graph.
A graph $G$ is not extremal if any two half-sized sets have many edges between them.

\begin{defn}[nice, extremal]
Let $G$ be a graph on a vertex set $V$ of size $n$ and let $\eps>0$. We say that
\begin{itemize}
\item $G$ is $\eps$-\emph{nice} if for any two 
sets $A,B \subseteq V$ of size at least $(\frac{1}{2}-\eps)n$, we have $e_G(A,B) \geq \eps n^2$.
\item $G$ is \emph{$\eps$-extremal} if it  is not $\eps^3$-nice. 
\end{itemize}
\end{defn}
Note that whenever $\eps' \geq \eps > 0$, an $\eps'$-nice graph is $\eps$-nice 
and hence an $\eps$-extremal graph is $\eps'$-extremal.

The following is a version of a well-known fact about the structure of almost Dirac graphs
which forms the basis of our extremal case distinction (for example, see~\cite{KLS}).

\begin{lemma}\label{lm:char}
Suppose that $0 < 1/n \ll d \ll \eps \leq 1$.
Let $G$ be a graph on a vertex set $V$ of size $n$ with $d_G(x) \geq (\frac{1}{2}-\eps^3)n$
for all but at most $d n$ vertices $x \in V$
which is $\eps$-extremal. 
Then there is a \emph{characteristic partition} $(A,B,C)$ of $G$ such that
\begin{itemize}
\item[(i)] $A,B,C$ partition $V$;
\item[(ii)] $|A|=|B|= (\frac{1}{2}-\eps)n$;
\item[(iii)] one of the following holds:
\begin{itemize}[$\bullet$]
\item $d_G(a,A)\geq (\frac{1}{2}-2\eps)n$ for all $a \in A$ and $d_G(b,B)\geq (\frac{1}{2}-2\eps)n$ for all $b \in B$ and $e_G(A,B) \leq \eps n^2$; here we say that $G$ is \emph{$(\eps,\ECone)$-extremal};
\item $d_G(a,B)\geq (\frac{1}{2}-2\eps)n$ for all $a\in A$ and $d_G(b,A)\geq (\frac{1}{2}-2\eps)n$ for all $b \in B$, and either $e_G(A) \leq \eps n^2$ or $e_G(B) \leq \eps n^2$; here we say that $G$ is \emph{$(\eps,\ECtwo)$-extremal}.
\end{itemize}
\end{itemize}
\end{lemma}

\begin{proof}
By adding vertices with small degree to $C$, it suffices to prove that if $G$ has $\dD(G) \geq (\frac{1}{2}-\eps^3)n$,
then the conclusion holds with $\eps/2$ in place of $\eps$.
Let $\mu := \eps^3$.
Since $G$ is $\eps$-extremal, there are $X,Y \subseteq V$ each of size at least $(\frac{1}{2}-\mu)n$ such that $e_G(X,Y) < \mu n^2$.
Let $U:=X\cap Y$ and $D:=V\sm(X\cup Y)$. We divide the proof into two cases based on the size of $U$.

\medskip
\noindent
\emph{Case 1}. $|U|\geq 2\sqrt{\mu} n$.
Let $U_0 := \{u \in U: d_G(u,D) \leq (\frac{1}{2}-3\sqrt{\mu})n\}$.
If $|U_0|> \sqrt{\mu}n$, then since $d_G(v,X\cup Y)\geq 2\sqrt{\mu}n$ for each $v\in U_0$, we have 
$e_G(X,Y) \geq e_G(U,X \cup Y) \geq \frac{1}{2}\cdot 2\sqrt{\mu}n\cdot \sqrt{\mu}n=\mu n^2$, a contradiction.
Thus we have $|U_0|\leq \sqrt{\mu}n$ and hence $|U\sm U_0|\geq \sqrt{\mu}n$. 
In particular, $U \sm U_0 \neq \es$ and thus $|D|\geq (\frac{1}{2}-3\sqrt{\mu})n$, so $|X\cup Y|\leq (\frac{1}{2}+3\sqrt{\mu})n$. 
We have
$$
|U \sm U_0|=|X|+|Y|-|X \cup Y|-|U_0| \geq 2(\tfrac{1}{2}-\mu)n-(\tfrac{1}{2}+3\sqrt{\mu})n-\sqrt{\mu}n \geq (\tfrac{1}{2}-5\sqrt{\mu})n.
$$
Each vertex in $U\sm U_0$ has at least $(\frac{1}{2}-3\sqrt{\mu})n$ neighbours in $V \sm U$. 
Since $d_G(v,D) \geq (\frac{1}{2}-3\sqrt{\mu})n$ for all $v \in U \sm U_0$,
we have $e_G(U \sm U_0,D) \geq (\frac{1}{2}-5\sqrt{\mu})(\frac{1}{2}-3\sqrt{\mu})n^2$ and hence 
$D_0 := \{x \in D: d_G(x,U \sm U_0) \leq (\frac{1}{2}-3\sqrt{\mu})n\}$ has size $|D_0| \leq 6\sqrt{\mu}n$.
So
$
|D \sm D_0| \geq (\tfrac{1}{2}-9\sqrt{\mu})n
$.

Now choose any $A \subseteq U \sm U_0$ and $B \subseteq D \sm D_0$ with $|A|=|B|=(\frac{1}{2}-\eps/2)n$,
and let $C := V \sm (A \cup B)$.
Then $A$ and $B$ are disjoint, and
we have $d_G(a,B) \geq d_G(a,D \sm D_0) - (\eps/2+5\sqrt{\mu}) n \geq (\frac{1}{2}-\eps)n$ for all $a \in A$
and similarly for $d_G(b,A)$ for all $b \in B$.
Finally, $e_G(A) \leq e_G(U) \leq e_G(X,Y) < \mu n^2 < \eps n^2/2$.
Thus $G$ is $(\eps/2,\ECtwo)$-extremal.

\medskip
\noindent
\emph{Case 2}. $|U|< 2\sqrt{\mu} n$.
This case is very similar so we only sketch the proof.
It is easy to see that for all but most $2\sqrt{\mu}n$ vertices $v$ in $X$ we have $d_G(v,Y)\leq \sqrt{\mu}n$. Similarly, for all but most $2\sqrt{\mu}n$ vertices $v$ in $Y$ we have $d_G(v,X)\leq \sqrt{\mu}n$. We delete these exceptional vertices from $X$ and from $Y$.
Let $X_1:=X\sm Y$ and $Y_1:=Y\sm X$. Note that these sets are disjoint and each has size at least $(\frac{1}{2}-5\sqrt{\mu})n$. Now it follows that for each vertex $v\in X_1$, we have $d_G(v,X_1)\geq (\frac{1}{2}-8\sqrt{\mu})n$ and for each vertex $v\in Y_1$, we have $d_G(v,Y_1)\geq (\frac{1}{2}-8\sqrt{\mu})n$.
We choose $A \subseteq X_1$ and $B \subseteq Y_1$ of the correct sizes, and note that $e_G(A,B) \leq e_G(X,Y) \leq \mu n^2$.
Thus $G$ is $(\eps/2,\ECone)$-extremal.
\end{proof}

\subsection{Stable graph collections}\label{sec:stabcol}

We now move on to consider stability for graph collections.
Lemma~\ref{lm:char} implies that given a positive integer $n$ and $0<1/n \ll \eps \leq 1$ and a graph collection $\bm{G} =(G_1,\ldots,G_n)$ on a common vertex set $V$ of size $n$, whenever $G_i$ is $\eps$-extremal, we can fix a
$$
\text{characteristic partition }(A_i,B_i,C_i) \text{ of }G_i.
$$
We say that a vertex $v \in V$ is \emph{$i$-good} if either $G_i$ is not $\eps$-extremal 
(that is, $G_i$ is $\eps^3$-nice)
or $G_i$ is $\eps$-extremal 
and $v \in A_i\cup B_i$. 
Extremality of a graph collection depends on where graphs are in relation to one another,
hence we make the following definitions.

\begin{defn}[crossing, similar, cross graph]\label{def:cross}
Let $0<1/n,\eps,\dD<1$ where $n \in \mb{N}$ and let $\bm{G}=(G_1,\ldots,G_n)$ be a graph collection on a common vertex set $V$ of size $n$.
Given $i,j \in [n]$ such that $G_i$ and $G_j$ are both $\eps$-extremal, we say that they are \emph{$\dD$-crossing} if $|A_i \sd A_j| \geq \dD n$ and $|A_i \sd B_j| \geq \dD n$.
We define the \emph{cross graph} $C^{\eps,\dD}_{\bm{G}}$ to be the graph with vertex set $[n]$ where $i$ is adjacent to $j$ if and only if $G_i, G_j$ are both $\eps$-extremal and $\dD$-crossing. 
\end{defn}

\begin{obs}\label{obs:cross}
Suppose that $0<\eps \leq \dD/8$.
If $G_i$ and $G_j$ are $\dD$-crossing, then $|X_i\cap Y_j|\geq \dD n/4$ whenever $X,Y\in\{A,B\}$.
\end{obs}
\begin{proof}
Since $|A_i\sd B_j|\geq \dD n$ and $|A_i|=|B_j|$ we get $|A_i\sm B_j|\geq \dD n/2$. Thus $|A_i\cap A_j|=|A_i\sm (B_j\cup C_j)|\geq \dD n/2 - 2\eps n \geq \dD n/4$.
The other assertions hold by symmetry.
\end{proof}

We can now define stability for graph collections.

\begin{defn}[strongly stable, weakly stable, stable,nice]\label{def:stab}
Let $n \in \mb{N}$ and $0<\mu,\gG,\aA,\eps,\dD<1$ be parameters. Suppose that $\bm{G}=(G_1,\ldots,G_n)$ is a graph collection on a common vertex set $V$ of size $n$. 
We say that
\begin{itemize}
\item $\bm{G}$ is $(\gG,\aA)$-\emph{strongly stable} if $G_i$ is $\aA$-nice for at least $\gG n$ colours $i \in [n]$;
\item $\bm{G}$ is $(\eps,\dD)$-\emph{weakly stable} if 
$e(C^{\eps,\dD}_{\bm{G}})\geq \dD n^2$;
\item $\bm{G}$ is $(\gG,\aA,\eps,\dD)$-\emph{stable} if it is either $(\gG,\aA)$-strongly stable or $(\eps,\dD)$-weakly stable (or both). 
\item $\bm{G}$ is \emph{$\mu$-nice} if 
for every $A \subseteq V$ of size $\lfloor n/2 \rfloor$ we have $e_{\bm{G}}(A)> \mu n^3$ and $e_{\bm{G}}(A,V \sm A)> \mu n^3$.
\end{itemize}
When the parameters are clear from the text, we simply say $\bm{G}$ is \emph{strongly stable}, \emph{weakly stable}, \emph{stable} or~\emph{nice}.
\end{defn}
Note that whenever $\gG' \geq \gG>0$ and $\aA' \geq \aA>0$,
a $(\gG',\aA')$-strongly stable collection is $(\gG,\aA)$-strongly stable,
and whenever $0< \eps' \leq \eps$ and $\dD' \geq \dD >0$,
an $(\eps',\dD')$-weakly stable collection is $(\eps,\dD)$-weakly stable. 

\begin{lemma}\label{lm:re-close}
Suppose that $0<1/n \ll \mu \ll \gG,\aA,\eps,\dD<1$. 
Let $\bm{G}=(G_1,\ldots,G_n)$ be a graph collection on a common vertex set $V$ of size $n$.
If $\bm{G}$ is $(\gG,\aA,\eps,\dD)$-stable, then $\bm{G}$ is $\mu$-nice.
\end{lemma}
\begin{proof}
Suppose that there is $A \subseteq V$ of size $\lfloor n/2 \rfloor$ such that $e_{\bm{G}}(A)\leq \mu n^3$.
Therefore, for all but at most $\sqrt{\mu}n$ colours $i \in [n]$, we have $e_{G_i}(A)\leq\sqrt{\mu}n^2$, and hence $G_i$ is not $\sqrt{\mu}$-nice. This implies that $\bm{G}$ is not $(1-\sqrt{\mu},\sqrt{\mu})$-strongly stable and hence not $(\gG,\aA)$-stable.
Without loss of generality, suppose that for each $i\in [(1-\sqrt{\mu})n]$, $G_i$ is $\eps$-extremal with characteristic partition $(A_i,B_i,C_i)$.
Let $B := V \sm A$.
It follows that, after possibly relabelling $A_i$ and $B_i$, 
we have $|A\sd A_i|\leq \dD^2 n$ and $|B\sd B_i|\leq \dD^2 n$ for all but at most $\dD^2 n$ colours $i\in [(1-\sqrt{\mu})n]$, otherwise we have at least $\dD^4n^3/2>\mu n^3$ edges of $\bm{G}$ inside $A$ by $\mu \ll \dD$, a contradiction. Delete such colours from $[(1-\sqrt{\mu})n]$ and let $\Cols$ be the remaining colour set.
Thus for every $i,j \in \Cols$, the graphs $G_i$ and $G_j$ are not $\dD^2$-crossing. Therefore 
$e(C^{\eps,\dD}_{\bm{G}})\leq \sqrt{\mu} n^2+\dD^2 n^2< \dD n^2$ by $\mu\ll \dD$. 
Thus $\bm{G}$ is not $(\eps,\dD)$-weakly stable, a contradiction.

The proof of the other case is similar and we omit it. 
\end{proof}

The next key lemma guarantees that a nice graph collection 
in which all vertices have degree at least $(\frac{1}{2}-\mu)n$ in \emph{almost all} colours
contains a transversal copy of a graph $H$ which \emph{locally} looks like a Hamilton cycle:
almost all degrees equal two.
We choose $H$ to be a union of two almost perfect matchings.
It is vital that here, `almost all' means a $1-\theta$ proportion
where $\theta \ll \mu$. 
If we were allowed to take, say, $\theta=3\mu$, then the very first part of the argument of Joos and Kim on the transversal minimum degree threshold for matchings~\cite{JK} would give a very short proof even without assuming the graph collection is stable.
Instead, we need to analyse the structure of maximal matchings in a collection of almost Dirac graphs.

This lemma will be applied in a reduced graph, which is why we only assume a weaker degree condition, and then the transversal blow-up lemma will be applied
to obtain a transversal collection of long paths covering almost the whole of $V$.

\begin{lemma}\label{lm:2matching}
Let $0<1/n\ll d\ll \mu,\theta \ll \gG,\aA,\eps,\dD \ll 1$. Suppose that $\bm{G}=(G_1,\ldots,G_n)$ is a graph collection defined on a common vertex set $V$ of size $n$ and
for each vertex $v \in V$,  we have $d_{G_i}(v)\geq (\frac{1}{2}-\mu)n$ for all but at most $dn$ colours $i\in [n]$. If $\bm{G}$ is $(\gG,\aA,\eps,\dD)$-stable, then $\bm{G}$ contains 
two edge-disjoint transversal matchings $M_1$ and $M_2$ in $\bm{G}$ such that
$e(M_\ell)\geq (\frac{1}{2}-\theta)n$ for $\ell=1,2$ and $\col(M_1)\cap \col(M_2)=\es$.
\end{lemma}

The lemma is an easy consequence of the next lemma which guarantees an almost spanning transversal matching in a nice graph collection containing about $n/2$ graphs.

\begin{lemma}\label{lm:1matching}
Let $0<1/n\ll d\ll \mu,\theta \ll 1$. Suppose that $\bm{G}=(G_1,\ldots,G_{(\frac{1}{2}-\theta/4)n})$ is a graph collection defined on a common vertex set $V$ of size $n$ and
for each vertex $v \in V$, we have $d_{G_i}(v)\geq (\frac{1}{2}-\mu)n$ for all but at most $dn$ colours $i \in [n]$.
If $\bm{G}$ is $100\mu$-nice, then $\bm{G}$ contains a transversal matching $M$ with
$e(M)\geq (\frac{1}{2}-\theta)n$.
\end{lemma}

\begin{proof}[Proof of Lemma~\ref{lm:2matching} given Lemma~\ref{lm:1matching}]
Let $\Cols_1$ be a random set obtained by selecting each colour in $[n]$ randomly and independently with probability $\frac{1}{2}$, and let $\Cols_2:=[n]\sm \Cols_1$.
\begin{claim}
 With high probability, for $\ell=1,2$ we have $|\Cols_\ell|=n/2\pm n^{3/4}$ and the graph collection $\bm{G}^\ell :=(G_i:i\in \Cols_{\ell})$ is 
$\mu$-nice.
\end{claim}
\begin{proof}[Proof of Claim]\label{cl:randomp}
The first statement follows from Chernoff's inequality (Lemma~\ref{lm:chernoff}). 

Suppose first that $\bm{G}$ is $(\gG,\aA)$-strongly stable.
Then $\ms{A}:=\{i\in [n]: G_i\text{ is }\aA\text{-nice}\}$ satisfies $|\ms{A}| \geq \gG n$.
By Chernoff's inequality, with high probability, we have $|\ms{A}\cap \Cols_{\ell}|\geq \gG n/3$ and thus $\bm{G}^\ell$ is $(\gG/2,\aA)$-strongly stable for $\ell=1,2$.

Suppose instead that $\bm{G}$ is $(\eps,\dD)$-weakly stable.
Let $\ms{A}:=\{i\in [n]: G_i\text{ is }\eps\text{-extremal}\}$. We have $e(C^{\eps,\dD}_{\bm{G}})\geq \dD n^2$, so in particular, $|\ms{A}| \geq \dD n$. Let $H:=C^{\eps,\dD}_{\bm{G}}$ and $\ms{B}:=\{i\in \ms{A}: d_{H}(i,\ms{A})\geq \dD n\}$. It follows that $|\ms{B}|\geq \dD n$. Let $\mc{B}:=\{N_{H}(v,\ms{A}): v\in \ms{B}\}$. By Chernoff's inequality, with high probability, we have $|\ms{B}\cap \Cols_{\ell}|\geq \dD n/3$ and $|N\cap \Cols_{\ell}|\geq \dD n/3$ for each $N\in \mc{B}$ and $\ell=1,2$. Thus with high probability, we have $e(H[\Cols_{\ell}])\geq \dD^2 n^2/18$ and thus $\bm{G}^\ell$ is $(\eps,\dD^2/18)$-weakly stable for $\ell=1,2$.

In both cases, Lemma~\ref{lm:re-close} implies that $\bm{G}^\ell$ is $\mu$-nice for $\ell=1,2$. 
\end{proof}

Fix such $\Cols_1$ and $\Cols_2$. 
Apply Lemma~\ref{lm:1matching} to $\bm{G}^1$ to obtain a transversal matching
$M_1$ with $e(M_1) \geq (\frac{1}{2}-\theta)n$ and colours from $\Cols_1$.
Remove any $xy \in E(M_1)$ from $G^2_i$
for each $i \in \Cols_2$. Every vertex degree in every graph reduces by at most one,
so $\bm{G}^2$ is still $2\mu$-nice.
Apply Lemma~\ref{lm:1matching} to $\bm{G}^2$ to obtain a transversal matching $M_2$
with $e(M_2) \geq (\frac{1}{2}-\theta)n$ and colours from $\Cols_2$.
\end{proof}

\begin{proof}[Proof of Lemma~\ref{lm:1matching}]
Let $M$ be a maximal transversal matching in $\bm{G}$ and 
suppose for a contradiction that 
$e(M) < (\frac{1}{2}-\theta)n$.
Let $v_1,v_2\notin V(M)$ and $c_1,c_2\notin \col(M)$ be distinct such that $N_\ell := N_{G_{c_\ell}}(v_\ell)$ satisfies $|N_\ell| \geq (\frac{1}{2}-\mu)n$ for $\ell=1,2$. 
Let $U := V \sm (V(M) \cup \{v_1,v_2\})$ and $\Cols :=[(\frac{1}{2}-\theta/4)n]\sm(\col(M) \cup \{c_1,c_2\})$.
Given $x \in V(M)$, we write $x^+$ to denote the unique neighbour of $x$ in $M$
and given $A \subseteq V(M)$, we write $A^+:=\{x^+: x \in A\}$.
Given $A \subseteq V$, we write $\ov{A} := V \sm A$.
\begin{claim}\label{cl:partition1}
We have $N_\ell \subseteq V(M)$ and $(\frac{1}{2}-\mu)n \leq |N_\ell|\leq (\frac{1}{2} + 3\mu)n$ for $\ell=1,2$ 
and exactly one of the following holds:
\begin{itemize}
\item 
$|\ov{N_1} \sd N_2| \leq 6\mu n$,
and writing $M_\ell := M[N_\ell]$ and $X_\ell := V(M_\ell)$
for $\ell=1,2$, we have
$e(M_\ell) \geq (\frac{1}{4}-2\mu)n$
and $X_1 \cap X_2 = \emptyset$ and 
$e(G_c[X_1,X_2]) = 0$ 
for all $c \in \Cols$; 
\item 
$|N_1 \sd N_2| \leq 6\mu n$, 
and there is $X \subseteq \ov{N_1 \cup N_2} \cap V(M)$ with $X^+ \subseteq N_1 \cap N_2$
such that
$|\ov{N_1 \cup N_2} \sm X|,|(N_1 \cap N_2) \sm X^+| \leq 6\mu n$
and $|X| \geq (\frac{1}{2}-3\mu)n$ and 
$e(M[X^+]) = 0$
and $e(G_c[X]) = 0$ for all $c \in \Cols$.
\end{itemize}
\end{claim}

\begin{proof}[Proof of Claim]
If there is $y \notin V(M)$ with $v_\ell y \in G_{c_\ell}$ for some $\ell=1,2$
then we can add $v_\ell y$ to $M$ to get a larger matching.
Thus $N_\ell \subseteq V(M)$ for $\ell=1,2$.
Moreover,
we have $N_\ell^+\cap N_{3-\ell}=\es$ since otherwise we get a $3$-path $v_{3-\ell} ww^+ v_{\ell}$ with $\col(v_{3-\ell} w)=c_{3-\ell}$, $\col(v_{\ell} w^+)=c_{\ell}$ and $w w^+\in E(M)$. Replacing $w w^+$ by $v_{3-\ell} w$ and $w^+v_{\ell}$ gives a transversal matching that is larger than $M$, a contradiction. 
Let $S_\ell := E(M[N_\ell])$ and $T_\ell := E(M[N_\ell,V-N_\ell])$
and write $s_\ell:=|S_\ell|$ and $t_\ell:=|T_\ell|$.
Note that, since $N_\ell^+ \cap N_{3-\ell} = \es$, we have
\begin{equation}\label{eq:S}
S_1\cap S_2=S_1 \cap T_2 = T_1 \cap S_2 = \es.
\end{equation}
Since $N_\ell \subseteq V(M)$ we have
$$
2s_\ell+t_\ell = |N_\ell| \geq (\tfrac{1}{2}-\mu)n. 
$$

\medskip
\noindent
\emph{Case 1.} $t_1+t_2\leq 2\mu n$.
We will show that the first alternative holds.
We have $x \in N_2$ only if $x^+ \notin N_1$ so $|N_1 \cap N_2| \leq t_1+t_2 \leq 2\mu n$,
and $|\ov{N_1 \cup N_2}| \leq n-2(s_1+s_2) \leq 2\mu n+t_1+t_2 \leq 4\mu n$.
So $|\ov{N_1} \sd N_2| \leq 6\mu n$.
Also $|N_2| \leq |\ov{N_1}|+2\mu n \leq (\frac{1}{2}+3\mu)n$
so by symmetry we have $(\frac{1}{2}-\mu)n \leq |N_\ell| \leq (\frac{1}{2}+3\mu)n$
for $\ell=1,2$.
Next, $e(M_\ell)=s_\ell \geq \frac{1}{2}(|N_\ell|-t_\ell) \geq (\frac{1}{4}-2\mu)n$.
We have $X_\ell = V(S_\ell)$
and since $S_\ell$ is a matching, $S_1 \cap S_2=\emptyset$
implies that $X_1 \cap X_2 = \emptyset$.

Finally, suppose $E(G_c[X_1,X_2]) \neq \emptyset$ for some $c \in \Cols$.
Then there are $w_\ell w_\ell^+ \in E(M_\ell)$ for $\ell=1,2$
with $w_1^+w_2 \in G_c$.
Thus there is a transversal path $v_1w_1w_1^+w_2w_2^+v_2$
using $c_1,c_2,c$ and two colours from $M$.
Hence
we can replace $w_1w_1^+$ and $w_2w_2^+$ by $v_1w_1,w_1^+w_2,v_2w_2^+$
to obtain a larger transversal matching, a contradiction.

\medskip
\noindent
\emph{Case 2.} $t_1+t_2\geq 2\mu n$.
We will show that the second alternative holds.
Now~(\ref{eq:S}) implies that
\begin{align*}
|T_1\cap T_2|&=t_1+t_2-|T_1\cup T_2|\geq t_1+t_2-(e(M)-s_1-s_2)= \frac{t_1+t_2}{2}+\frac{2s_1+t_1}{2}+\frac{2s_2+t_2}{2}-e(M)\\
&\geq \mu n+(\tfrac{1}{2}-\mu)n-(\tfrac{1}{2}-\theta)n\geq \theta n.
\end{align*}
Since $N_\ell^+ \cap N_{3-\ell}=\emptyset$, we can choose an edge $v_3v_3^+\in T_1\cap T_2$ where $v_3^+\in N_1\cap N_2$ and $v_3\in \ov{N_1\cup N_2}$. Choose any colour $c_3\in \Cols$ such that $d_{G_{c_3}}(v_3) \geq (\frac{1}{2}-\mu)n$. Let $N_3:=N_{G_{c_3}}(v_3)$.

Suppose $N_3\cap N_\ell^+\neq \es$ for $\ell \in [2]$. Thus there exists $ww^+\in E(M)$ such that $\col(v_\ell w^+)=c_\ell$, $\col(v_3w)=c_3$ and hence $v_\ell w^+wv_3v_3^+v_{3-\ell}$ is a transversal path with $v_\ell w^+, wv_3, v_3^+v_{3-\ell}$ coloured by $c_\ell, c_3, c_{3-\ell}$. By replacing $v_3v_3^+, ww^+$ in $M$ by $v_3^+v_{3-\ell}, v_\ell w^+,wv_3$, we get a larger transversal matching, a contradiction. 
Thus $N_3 \cap N_\ell^+ = \es$.
Also, $N_3 \subseteq V(M)$:
if not, there is $w \notin V(M)$ such that 
we can replace $v_3v_3^+$ in $M$ with $v_3w, v_3^+v_1$ using colours $c_3,c_1$.
We have shown that $N_3, N_1^+ \cup N_2^+$ are pairwise disjoint subsets of $V(M)$.
By definition, we have the partition
$$
N_1^+ \cup N_2^+ = V(S_1) \cup V(S_2) \cup Y
\quad\text{where}\quad
Y:=\{x^+: x \in N_1, xx^+ \in T_1\} \cup \{x^+: x \in N_2, xx^+ \in T_2\}.
$$
Whenever $x \in Y$ we have $x^+ \notin Y$.
So $|Y| \geq |T_1 \cup T_2|$ and hence
$|V(M)| \geq |N_3|+2s_1+2s_2+|T_1 \cup T_2|$.
Therefore
$$
|T_1\cap T_2| = t_1+t_2-|T_1 \cup T_2|
\geq \sum_{\ell=1,2}(2s_\ell+t_\ell)+|N_3|-|V(M)|\geq 3(\tfrac{1}{2}-\mu)n-n\geq (\tfrac{1}{2}-3\mu)n.
$$
Let $X := \{x \in \ov{N_1 \cup N_2}: xx^+ \in T_1 \cap T_2\}$. 
So $X \subseteq \ov{N_1 \cup N_2}$ and $X^+ \subseteq N_1 \cap N_2$
are disjoint, so $|X|=|X^+| \geq |T_1 \cap T_2|$.
Thus also $|\ov{N_1 \cup N_2} \sm X|,|(N_1 \cap N_2) \sm X^+| \leq n-|X|-|X^+| \leq 6\mu n$.
Since $N_1 \cup N_2$ and $X$ are disjoint, we have $|N_\ell| \leq |N_1 \cup N_2| \leq (\frac{1}{2}+3\mu)n$
for $\ell=1,2$.
We have $|N_1 \sd N_2| = |N_1 \cup N_2| - |N_1 \cap N_2| \leq n-|X^+|-|X| \leq 6\mu n$.
The definition of $X$ implies that $E(M[X^+])=\emptyset$.

Finally, suppose $e(G_c[X]) \neq 0$ for some $c \in \Cols$.
Then as before we obtain a contradiction by adding such an edge $xy$ along with $v_1x^+,v_2y^+$
(of colours $c,c_1,c_2$) and removing $xx^+,y^+y$.
This completes the proof of the claim.
\end{proof}
By Claim \ref{cl:partition1}, we have 
$N_\ell \subseteq V(M)$ and $(\frac{1}{2}-\mu)n \leq |N_\ell|\leq (\frac{1}{2} + 3\mu)n$ for $\ell=1,2$ 
and one of the following cases:

\medskip
\noindent
\emph{Case 1}: $|\ov{N_1} \sd N_2| \leq 6\mu n$,
and writing $M_\ell := M[N_\ell]$ and $X_\ell := V(M_\ell)$
for $\ell=1,2$, we have
$e(M_\ell) \geq (\frac{1}{4}-2\mu)n$
and $X_1 \cap X_2 = \emptyset$ and 
$e(G_c[X_1,X_2]) = 0$ 
for all $c \in \Cols$.

It suffices to show that 
for all colours $c \in \col(M_1 \cup M_2)$, we have
$e_{G_c}(X_1,X_2) \leq 13\mu n^2$.
Indeed, there are at least $2(\frac{1}{4}-2\mu)n = (\frac{1}{2}-4\mu)n$ such colours
so $e_{\bm{G}}(X_1,X_2) \leq 13\mu n^3 + 4\mu n^3=17\mu n^3$.
Also $X_1,X_2$ are disjoint sets of size at least $(\frac{1}{2}-4\mu)n$,
so there is a set $Y$ of size $n/2$
such that $e_{\bm{G}}(Y,\ov{Y}) \leq 30 \mu n^3$,
and therefore $\bm{G}$ is not $30\mu$-stable.

Now let $c \in \col(M_1 \cup M_2)$.
Without loss of generality, there is $ww^+\in E(M_1)$ such that 
$\col(ww^+)=c$. Choose a colour $c_3\in \Cols
$ 
so that $N_3 := N_{G_{c_3}}(w)$
satisfies $|N_3| \geq (\frac{1}{2}-\mu)n$.
So $N_3 \cap X_2 = \emptyset$.
Now let $M':=M-ww^+ + v_1w^+$ 
be the transversal matching with $\col(M') = (\col(M) \sm \{c\}) \cup \{c_1\}$
and consider new vertex pair $(w,v_2)$ and colour pair $(c_3,c_2)$. We have $\{c,c_2,c_3\} \cap \col(M')=\emptyset$. 
The matching $M'$ is maximal, otherwise the maximality of $M$ is contradicted.
So Claim \ref{cl:partition1} applies to $M'$.

Now, $N_3 \cap X_2 = \emptyset$ so $|N_3 \cap N_2| \leq |N_2|-|X_2| \leq (\frac{1}{2}+3\mu)n-(\frac{1}{2}-4\mu)n=7\mu n$
and in particular the second alternative cannot hold. 
Thus
$|\ov{N_3} \sd N_2| \leq 6\mu n$,
and writing $M_\ell' := M'[N_\ell]$ and $X_\ell' := V(M_\ell')$
for $\ell=2,3$, we have
$e(M_\ell') \geq (\frac{1}{4}-2\mu)n$
and $X_2' \cap X_3' = \emptyset$ and 
$E(G_c[X_2',X_3']) = \emptyset$.
We have $|N_1 \sd N_3| \leq |N_1 \sd \ov{N_2}|+|\ov{N_2} \sd N_3| \leq 12\mu n$. Since $M,M'$ differ by two edges we therefore have $|X_1 \sd X_3'| + |X_2 \sd X_2'| \leq 13\mu n$.
Thus $e_{G_c}(X_1,X_2) \leq e_{G_c}(X_3',X_2') \leq 13\mu n^2$,
as required.

\medskip
\noindent
\emph{Case 2}: 
$|N_1 \sd N_2| \leq 6\mu n$, 
and there is $X \subseteq \ov{N_1 \cup N_2} \cap V(M)$ with $X^+ \subseteq N_1 \cap N_2$
such that
$|\ov{N_1 \cup N_2} \sm X|,|(N_1 \cap N_2) \sm X^+| \leq 6\mu n$
such that $|X| \geq (\frac{1}{2}-3\mu)n$ and 
$e(M[X^+]) = 0$
and $e(G_c[X]) = 0$ for all $c \in \Cols$.

It suffices to show that $e_{G_c}(X) \leq 24\mu n^2$ for all $c \in \col(M[X,X^+])$.
Indeed, there are $|X| \geq (\frac{1}{2}-3\mu)n$
such colours, so $e_{\bm{G}}(X) \leq 24\mu n^3+3\mu n^3=27\mu n^3$.
Since $(\frac{1}{2}-3\mu)n \leq |X| \leq n-|X^+| \leq (\frac{1}{2}+3\mu)n$, there is a set $Y$ of size $n/2$ with $e_{\bm{G}}(Y) \leq 40\mu n^3$.
Thus $\bm{G}$ is not $40\mu$-stable.

Let $c \in \col(M[X, X^+])$.
Let $w \in X$ be such that $\col(ww^+)=c$.
Choose a colour $c_3$ from $\Cols$ such that $N_3 := N_{G_{c_3}}(w)$ satisfies $|N_3| \geq (\frac{1}{2}-\mu)n$. 
So $N_3 \cap X = \emptyset$.
Now let $M':=M-ww^++v_1w^+$ 
be the transversal matching with $\col(M')=(\col(M \sm \{c\}) \cup \{c_1\}$
and consider new vertex pair $(w,v_2)$ and colour pair $(c_3,c_2)$.
So $\{c,c_2,c_3\} \cap \col(M')=\emptyset$. 
The matching $M'$ is maximal, otherwise the maximality of $M$ is contradicted.
So Claim \ref{cl:partition1} applies to $M'$.
We have $N_3 \cup N_2 \subseteq \ov{X}$ so $|N_3 \cap N_2| \geq |N_3|+|N_2|-(n-|X|) \geq 2(\frac{1}{2}-\mu)n-n+|X| >(\frac{1}{2}-5\mu)n$,
so in particular the first alternative cannot hold.

So $|N_3 \sd N_2| \leq 6\mu n$,
and there is $X' \subseteq \ov{N_3 \cup N_2}$ with 
$|\ov{N_3 \cup N_2} \sm X'| \leq 6\mu n$
such that $|X'| \geq (\frac{1}{2}-3\mu)n$ and 
$e(M[(X')^+]) = 0$
and $e(G_c[X']) = 0$.
Now,
\begin{align*}
|X' \sd X| &\leq |X' \sd \ov{N_3 \cup N_2}| + |\ov{N_3 \cup N_2} \sd \ov{N_1 \cup N_2}| + |\ov{N_1 \cup N_2} \sd X| \leq 12\mu n + |N_3 \sd N_1|\\
&\leq 12\mu n+|N_3 \sd N_2|+|N_2 \sd N_1| \leq 24\mu n.
\end{align*}
Thus $e_{G_c}(X) \leq 24\mu n^2$, as required.

This completes the proof of Lemma \ref{lm:1matching}.
\end{proof}

\section{Absorbing}\label{sec:absorb}

Recall that $(V)_k$ denotes the set of all $k$-tuples of distinct elements of $V$.
A pair $H=(V,E)$ where $V$ is a set and $E \subseteq (V)_k$ is called a \emph{directed $k$-graph},
while if we allow $E$ to be a multiset, it is a \emph{directed multi-$k$-graph}.
We use the same notation for directed $k$-graphs as for graphs.
The following lemma is the key tool that we will use to build an absorbing structure. Similar ideas have already been used in the papers \cite{Cheng2,Cheng3,Cheng1} of the first author.

\begin{lemma}\label{lm:key}
Let $k,C,n\in \mathbb N$ suppose that $0<1/n\ll \gG \ll \eps \ll 1/k,1/C \leq 1$ and let $m\in [n^C]$ and $t := \gG n$. Let $\bm{H}=(H_1,\ldots,H_t)$ be a collection of directed $k$-graphs and let $\bm{Z}=(Z_1,\ldots,Z_m)$ be a collection of directed multi-$k$-graphs all defined on a common vertex set $V$ of size $n$.
Suppose that $e(H_i) \geq \eps n^k$ for all $i \in [t]$ and
for each $j\in [m]$, we have $|E(Z_j)\cap E(H_i)|\geq \eps n^k$ for at least $\eps t$ indices
 $i\in[t]$.
Then there is a transversal matching $M$ in $\bm{H}$ 
of size at least $(1-\eps^2/4)t$ and $|E(Z_j)\cap E(M)|\geq \eps^2 t/4$ for each $j\in [m]$.
\end{lemma}

\begin{proof}
Let $W=\{e_1,\ldots, e_t\}$ be obtained by independently picking $e_i$ from $E(H_i)$, $i\in [t]$, uniformly at random.
For each $i \in [t]$ and $j \in [m]$, let $\chi_{ij}=\ind(e_i\in Z_j)$. By definition, we have $e(H_i)\leq n^k$ for each $i\in[t]$ and thus $\mb{E}(\chi_{ij})=\mb{P}(e_i\in Z_j)=|E(Z_j)\cap E(H_i)|/|E(H_i)|\geq \eps n^k/n^k=\eps$ whenever $|E(Z_j)\cap E(H_i)|\geq \eps n^k$. Therefore
$\mb{E}(|W \cap E(H_j)|)=\mb{E}(\sum_{i \in [t]}\chi_{ij}) = \sum_{i\in[t]}\mathbb{E}(\chi_{ij})\geq \eps^2 t$
for each $j\in [m]$.
Chernoff's inequality and a union bound imply that with probability at least $1-e^{-\sqrt{n}}$, we have
$|W \cap E(H_j)| \geq \eps^2t/2$ for all $j \in [m]$.

Let $Y$ be the number of intersecting pairs of edges in $W$.
So $Y\leq\sum_{ii'\in\binom{[t]}{2}}Y_{ii'}$ where $Y_{ii'}=\ind(V(e_i)\cap V(e_{i'})\neq\es)$.
Let $I_{ii'}$ be the number of intersecting pairs of edges $\{a_i,a_{i'}\}$ with $a_h \in E(H_h)$ for $h=i,i'$.
So $|I_{ii'}|\leq kn^{2k-1}$.
Hence, for every $ii'\in\binom{[t]}{2}$, we have
$$
\mathbb{E}(Y_{ii'})=\frac{|I_{ii'}|}{e(H_i)e(H_{i'})}\leq \frac{kn^{2k-1}}{\eps^2n^{2k}} = \frac{k}{\eps^2 n}
\quad\text{and so}\quad
\mathbb{E}(Y)\leq\sum_{ii'\in\binom{[t]}{2}}\mathbb{E}(Y_{ii'})\leq \binom{t}{2}\frac{k}{\eps^2 n}\leq\frac{k}{2\eps^2}\gG^2n\leq \frac{\eps^2t}{8}.
$$
Markov's inequality implies that $\mathbb{P}(Y\leq \eps^2t/4)\geq 1/2$.
Thus a union bound implies that there exists $W$ with $|W|=t$, $|W \cap E(Z_j)| \geq \eps^2t/2$ for all $j \in [m]$ and there are at most $\eps^2t/4$ pairs of intersecting edges in $W$.
The required $M$ is obtained by
deleting one edge from each intersecting pair of $W$.
\end{proof}

Let $\bm{G}=(G_1,\ldots,G_n)$ be a graph collection on a common vertex set $V$ of size $n$.
For every pair $x,y$ of distinct vertices in $V$, we define $L(xy):=\{i \in [n] : xy\in E(G_i)\}$ to be the set of colours appearing on $xy$.
Let $\eps>0$ and suppose that $\dD(G) \geq (\frac{1}{2}-\eps^3)n$.
Recall from Lemma~\ref{lm:char} and Section~\ref{sec:stabcol} that
each $G_i$ is either $\eps^3$-nice or we have defined a characteristic partition $(A_i,B_i,C_i)$ of $G_i$.
Thus for every $i \in [n]$ and $v \in V$, either $v$ is $G_i$-good, or not. 

\begin{figure}
\begin{center}
\includegraphics{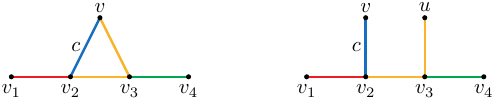}
\caption{A $c$-absorbing path of $(v,v)$ and a $c$-absorbing path of $(v,u)$ for $v \neq u$.}\label{fig:abs}
\end{center}
\end{figure}

\begin{defn}[absorbing path, absorbing cycle]\label{def:absorb}

Given any two not necessarily distinct vertices $u,v \in V$, a transversal path $P=v_1v_2v_3v_4$ with $u,v \notin V(P)$ is called a \emph{$c$-absorbing path of $(v,u)$} if $c\in L(v_2v)$ and $\col(v_2v_3)\in L(v_3u)$
(see Figure~\ref{fig:abs}).

Given  $\dD,\dD',\gG,\gG'>0$, a transversal cycle $C=v_1v_2\ldots v_{t}$ is an \emph{absorbing cycle} with parameters $(\dD,\dD',\gG,\gG')$ if $t\leq \gG n$ and there exists a colour set $\Cols$ of size at least $\dD n$ such that
\begin{itemize}
\item[(i)] given any colour $c\in \Cols$ and any $G_c$-good vertex $v$, for all but at most $\dD' n$ vertices $u\in V$, there are at least $\gG' n$ disjoint $c$-absorbing paths of $(v,u)$ inside $C$;
\item[(ii)] given any colour $c\in \Cols$, for all but at most $\dD' n$ $G_c$-good vertices $v$, there are at least $\gG' n$ disjoint $c$-absorbing paths of $(v,v)$ inside $C$.
\end{itemize}
\end{defn}

It is easy to see that if $C$ is a transversal cycle containing a $c$-absorbing path of $(v,v)$ where $v\notin V(C)$ and $c\notin \col(C)$, then there exists a transversal cycle $C'$ such that $V(C')=V(C)\cup \{v\}$ and $\col(C')=\col(C)\cup \{c\}$ (insert $v$ between $v_2$ and $v_3$). Similarly, if $C$ is a transversal cycle containing a $c$-absorbing path of $(v,u)$ where $u,v$ are the endpoints of a transversal path $P$, $V(P)\cap V(C)=\es$ and $\col(P),\{c\},\col(C)$ are pairwise disjoint, then there exists a transversal cycle $C'$ such that $V(C')=V(C)\cup V(P)$ and $\col(C')=\col(C)\cup \col(P)\cup \{c\}$ (insert $P$ between $v_2$ and $v_3$). 

Using Lemma~\ref{lm:key} it is quite straightforward to find an absorbing cycle when $G$ is strongly stable, using the many nice graphs $G_i$ guaranteed by strong stability.

\begin{lemma}\label{lm:absorbcycle}
Let $0<1/n\ll \lL,\mu \ll \gG,\aA\ll 1$ and 
suppose that $\bm{G}=(G_1,\ldots,G_n)$ is a graph collection defined on a common vertex set $V$ of size $n$ with $\dD(\bm{G})\geq (\frac{1}{2}-\mu)n$. If $\bm{G}$ is $(\gG,\aA)$-stable, then there exists an absorbing cycle with parameters $(1,0,\lL,\lL^2)$.
\end{lemma}

\begin{proof}
Without loss of generality, we assume that $G_1,\ldots,G_{\gG n}$ are $\aA$-nice and $6|\lL n$. We divide the colour set $[\lL n/2]$ into consecutive sets $\{1,2,3\},\{4,5,6\},\ldots,\{\lL n/2-2,\lL n/2-1,\lL n/2\}$ of
three. For each $i \in I := [\frac{\lL n}{6}]$, we define 
$$
\mc{F}_{i}:=\{(v_1,v_2,v_3,v_4) \in (V)_4: 3(i-1)+\ell \in L(v_\ell v_{\ell+1}) \ \forall \ell\in[3]\}. 
$$
Our minimum degree condition implies that $e(\mc{F}_i) \geq n((\frac{1}{2}-\mu)n)^3 \geq n^4/9$.
Now we fix a colour $c\in [n]$ and two not necessarily distinct vertices $v,v'\in V$
(which are not required to be $G_c$-good). We let $Z_i(c,vv')$ be the collection of $(v_1,v_2,v_3,v_4) \in \mc{F}_i$ for which $P=v_1v_2v_3v_4$ with the given colours is a $c$-absorbing path of $(v,v')$. Since the graph $G_{3i+2}$ is $\aA$-nice, there are at least $\aA n^2$ ways to choose $v_2v_3$ with $v_2\in N_{G_c}(v), v_3\in N_{G_{3i+2}}(v')$ and $v_2v_3\in E(G_{3i+2})$.
When $v_2,v_3$ are fixed, there are at least $(\frac{1}{2}-\mu)^2n^2$ ways to choose $v_1$ and $v_4$. Thus $|Z_i(c,vv')| \geq (\frac{1}{2}-\mu )^2n^2 \aA n^2\geq \aA n^4/5$.

We can apply Lemma~\ref{lm:key} with $\bm{H}:=(\mc{F}_i: i \in I)$,
$\bm{Z} :=(Z(c,vv') := \bigcup_{i \in I}Z_i(c,vv'): c \in [n], vv'\in V^2)$,
where this is a multiset union,
and parameters $|I|=t$ and $\aA/5=\eps$,
to obtain a transversal matching $M$ in $\bm{H}$ of size at least $(1-\aA^2/100)\lL n/6$
such that $|E(Z(c,vv')) \cap E(M)| \geq \aA^2\lL n/600$ for all $c \in [n]$ and $vv' \in V^2$.
That is, there is $I' \subseteq I$ with $|I'| \geq (1-\aA^2/100)\lL n/6$ such that for each $i \in I'$ there is
a transversal path $P_i=v_1^iv_2^iv_3^iv_4^i$ with $3(i-1)+\ell\in L(v^i_\ell v^i_{\ell+1})$ for all $\ell\in[3]$,
and for every $c \in [n]$ and $vv' \in V^2$, there are at least $\aA^2\lL n/600$ paths $P_i$
which are $c$-absorbing of $(v,v')$.

Relabel indices so that $I'=[s]$ where $s := (1-\aA^2/100)\lL n/6$.
We will connect the paths $P_1,\ldots,P_s$ one by one into a transversal cycle $C$. 
We first connect $P_1$ and $P_2$ into a single transversal path $P_1xyP_2$. 
For this, choose distinct unused colours $c_1,c_2\in [\lL n+1,n]$ and $c_3\in [\lL n/2+1,\lL n]$. Let $U:=V\sm V(\bigcup_{i \in [s]}V(P_i))$. Note that $d_{G_{c_1}}(v_4^{1},U),d_{G_{c_2}}(v_1^{2},U) \geq (\frac{1}{2}-\mu-\lL)n \geq (\frac{1}{2}-\aA)n$. 
Since $c_3\in [\lL n/2+1,\lL n]$, the graph $G_{c_3}$ is $\aA$-nice, so there are at least $\aA n^2$ $c_3$-coloured edges between $N_{G_{c_1}}(v_4^{1},U)$ and $N_{G_{c_2}}(v_1^{2},U)$. Thus we can choose a $c_3$-coloured edge $xy$ between $N_{G_{c_1}}(v_4^{1},U)$ and $N_{G_{c_2}}(v_1^{2},U)$ such that $C=P_1xyP_2$ is a transversal path with colour set $\col(P_1)\cup \col(P_2)\cup \{c_1,c_2,c_3\}$.

We repeat the process above for each pair $(P_i,P_{i+1})$ for $i\in[s]$, where $P_{s+1} := P_1$. Each time we use two unused colours from $[\lL n+1,n]$ and one unused colour from $[\lL n/2+1,\lL n]$. This is possible since $2s \leq \lL n/3$. The number of used vertices at each step is at most $6s \leq \lL n$ so we have the same bounds on degrees into the set of used vertices.
By construction, $C$ is an absorbing cycle with parameters $(1,0,6s/n,\aA^2\lL/600)$
and hence with parameters $(1,0,\lL,\lL^2)$ since $\lL \ll \aA$.
\end{proof}

To complete this section, we find an absorbing cycle when $\bm{G}$ is weakly stable.
This is much more involved than the strongly stable case.
Again we use Lemma~\ref{lm:key} to construct the cycle, using the many pairs of extremal graphs which are all highly crossing, as guaranteed by weak stability.

\begin{lemma}\label{lm:absorbcycle2}
Let $0<1/n \ll \lL,\mu,\eps \ll \dD<1$.
Suppose that $\bm{G}=(G_1,\ldots,G_n)$ is a graph collection defined on a common vertex set $V$ where $|V|=n$ and $\dD(\bm{G})\geq (\frac{1}{2}-\mu)n$.
If $\bm{G}$ is $(\eps,\dD)$-weakly stable, then there exists an absorbing cycle $C$ with parameters $(\dD/3,\sqrt{\eps},\lL,\lL^2)$.
\end{lemma}
\begin{proof}
For each $j\leq n/3$, define the following family.
$$
\mc{F}_j=\{(v_1,v_2,v_3,v_4) \in (V)_4: 3(j-1)+\ell \in L(v_\ell v_{\ell+1}) \ \forall \ell\in[3]\}.
$$
We first prove the following claim, which requires only the minimum degree condition on $\bm{G}$.
\begin{claim}\label{cl:num-absorb}
Let $ij \in E(C^{\eps,\dD}_{\bm{G}})$ where $j+1 \leq 3n$ 
and let $u$ be a $G_j$-good vertex and $v$ be a $G_i$-good vertex.
Let $Z_j(i,vu)$ be the collection of $(v_1,v_2,v_3,v_4) \in \mc{F}_{(j+1)/3}$ where $P=v_1v_2v_3v_4$ is an $i$-absorbing path of $(v,u)$.
Then $|Z_j(i,vu)| \geq 2^{-7}\dD^2n^4$.
\end{claim}
\begin{proof}[Proof of Claim]
Fix any such $i,j,u,v$.
By the definition of characteristic partition, we have $d_{G_i}(v,A_i)\geq (\frac{1}{2}-2\eps)n$ or $d_{G_i}(v,B_i)\geq (\frac{1}{2}-2\eps)n$. 
Without loss of generality, we assume that the former case holds since the latter case can be proved similarly. 
Since 
the graphs $G_i$ and $G_{j}$ are $\dD$-crossing, we have $|A_i\cap A_{j}|\geq \dD n/4$ and $|A_i\cap B_{j}|\geq \dD n/4$ by Observation \ref{obs:cross}. Since $d_{G_i}(v,A_i)\geq (\frac{1}{2}-2\eps)n$, we have $d_{G_i}(v,A_{j})\geq |A_i\cap A_{j}|-(|A_i|-d_{G_i}(v,A_i))\geq \dD n/4-\eps n\geq \dD n/5$ by $\eps \ll \dD$. Similarly, we have $d_{G_i}(v,B_{j})\geq \dD n/5$. 
Since $u$ is $G_{j}$-good, $u \in A_j \cup B_j$.
We divide the proof into the following cases:

\medskip
\noindent
\emph{Case 1}: $u\in A_{j}$ and $G_{j}$ is $(\eps,\ECone)$-extremal.
In this case, in the graph $G_{j}$, every vertex in $A_{j}$ is adjacent to all but at most $\eps n$ vertices in $A_{j}$. Thus 
$J := G_{j}[N_{G_i}(v,A_{j})]$ is an almost complete graph of order at least $\dD n/5$ in the sense that each vertex in it is adjacent to all but at most $\eps n$ vertices.
Any choice of $x \in V(J)$, $y \in N_{G_{j}}(u) \cap N_J(x)$, $x' \in N_{G_{j-1}}(x)$ and $y' \in N_{G_{j+1}}(y)$
yields an $i$-absorbing path $x'xyy'$ of $(v,u)$ in $\mc{F}_{(j+1)/3}$ (so its colours are, in order, $j-1,j,j+1$, and $i \in L(xv)$ and $j \in L(yu)$).
The number of such paths is therefore at least $\dD n/5 \cdot (\dD n/5-2\eps n) \cdot (\frac{1}{2}-\mu)^2n^2 \geq 2^{-7}\dD^2 n^4$.

\medskip
\noindent
\emph{Case 2}: $u\in A_j$ and $G_j$ is $(\eps,\ECtwo)$-extremal.
In this case, in the graph $G_j$, every vertex in $A_j$ is adjacent to all but at most $\eps n$ vertices in $B_j$ 
and every vertex in $B_j$ is adjacent to all but at most $\eps n$ vertices in $A_j$. 
Thus $J := G_j[N_{G_i}(v,A_j),N_{G_j}(u,B_j)]$ is a bipartite graph with each part of order at least $\dD n/5$ which is almost complete in the sense that each vertex is adjacent to all but at most $\eps n$ vertices from other side. 
Any choice of $x \in V(J) \cap A_j$, $y \in N_{G_j}(u) \cap N_J(x)$, $x' \in N_{G_{j-1}}(x)$ and $y' \in N_{G_{j+1}}(y)$
yields an $i$-absorbing path $x'xyy'$ of $(v,u)$ in $\mc{F}_{(j+1)/3}$.
The same lower bound from the previous case applies to the number of such choices.

The remaining cases when $u \in B_j$ are identical, so we omit them.
\end{proof}

For any vertex $v$, recall Definition~\ref{def:cross} and let 
$$
\ms{G}_v :=\{i\in [n]: v\text{ is }G_i\text{-good}\}.
$$
Since $e(C^{\eps,\dD}_{\bm{G}})\geq \dD n^2$, there is a subgraph $H$ of $C^{\eps,\dD}_{\bm{G}}$ such that $\dD(H)\geq \dD n$, so in particular, $|V(H)| > \dD n$.  
For each $i\in V(H)$, define the set $T_{i}:=\{u \in V: |N_{H}(i) \sm \ms{G}_u|\geq \dD n/2 \}$. Note that since $|C_i|=2\eps n$ for each $i$ where $G_i$ is $\eps$-extremal, we have $|T_{i}|\dD n/2\leq \sum_{u\in V}|[n] \sm \ms{G}_u| = \sum_{j \in [n]}|C_j| \leq 2\eps  n^2$ and thus $|T_{i}|\leq 4\eps n/\dD \leq \sqrt{\eps} n/2$ since $\eps \ll \dD$. Let $\ov{T}_{i}:=V\sm T_{i}$. For each $i\in V(H)$ and $u\in \ov{T}_{i}$, we have $|\ov{T}_{i}|\geq (1-\sqrt{\eps}/2)n$ and $|\ms{G}_u\cap N_H(i)|\geq \dD n/2$.

Now we independently and randomly select vertices from $V(H)$ with probability $\kK := \lL/14$ to obtain a set $\ms{U}$ of colours. 
A Chernoff bound implies that, with high probability,
\begin{itemize}
\item[(i)] $\kK |V(H)|/2 \leq t := |\ms{U}| \leq 2\kK |V(H)|$;
\item[(ii)] for every $i\in V(H)$ and $u\in \ov{T}_{i}$, we have $|\ms{G}_u\cap N_H(i,\ms{U})| \geq \dD \kK n/4$.
\end{itemize}
Fix such a $\ms{U}$.
By relabelling colours, we assume that $\ms{U}=\{3j-1: j \in [t]\}$.
Let $\ov{\ms{U}}:=V(H)\cap [3t+1,n]$. 
So $\ms{U} \cap \ov{\ms{U}} = \emptyset$ and $|\ov{\ms{U}}| \geq \dD n/2$ 
and $\dD(H[\ov{\ms{U}}]) \geq \dD n/2$ since $\lL \ll \dD$.
For each colour $j \in [t]$, let $I_j$ be the collection of $(i,vu) \in [n]\times V^2$
with $i \in \ov{\ms{U}}$, $u \in \ov{T}_i$, $j \in \ms{G}_u \cap N_H(i,\ms{U})$ and $v$ a $G_i$-good vertex.

Given $i$ and $u$, the number of choices of $j$ is at least $\dD\kK n/4$ by~(ii).
Since $i,j \in V(H)$ and $u$ is $G_j$-good, Claim~\ref{cl:num-absorb} implies that
there are at least $2^{-7}\dD^2 n^4$ $i$-absorbing paths of $(v,u)$ whose ordered vertex set is in $\mc{F}_{(j+1)/3}$.
We can apply Lemma~\ref{lm:key} with
$\bm{H} := \{\mc{F}_j: j \in [t]\}$ and the directed multi-$4$-graph collection
$$
\bm{Z} :=(Z(i,vu) := \textstyle\bigcup_{j \in \ms{G}_u \cap N_H(i,\ms{U})}Z_j(i,vu) : i \in \ov{\ms{U}}, u \in \ov{T}_i, v \text{ is }G_i\text{-good})
$$ 
and with $2^{-7}\dD^2$ playing the role of $\eps$.
Now, $e(\mc{F}_j) \geq n^4/9$ for all $j \in [t]$ and
$\bm{H}$ contains $\kK\dD n/2 \leq t \leq 2\kK n$ graphs by~(i), so $t/n \ll 2^{-7}\dD^2$.
For each $(i,vu)$ for which $Z(i,vu) \in \bm{Z}$ and every fixed $j \in \ms{G}_u \cap N_H(i,\ms{U})$,
we have $(j+1)/3 \in [t]$ and $|E(Z(i,vu)) \cap E(\mc{F}_{(j+1)/3})| \geq 2^{-7}\dD^2 n^4$.
The number of such $j$ is at least $\dD\kK n/4 \geq \dD t/8$ by~(ii).
Thus there is a transversal matching $M$ in $\bm{H}$ of size at least $(1-2^{-16}\dD^4)t$
such that $|E(Z(i,vu)) \cap E(M)| \geq 2^{-16}\dD^4 t$ for all $Z(i,vu) \in \bm{Z}$.
That is, there is $I \subseteq [t]$ with $|I| \geq (1-2^{-16}\dD^4)t$ such that for each $j \in I$
there is a path $P_j=v_1^jv_2^jv_3^jv_4^j$ in $\mc{F}_{(j+1)/3}$
and for every $i \in \ov{\ms{U}}$, $u \in \ov{T}_i$ and $v$ which is $G_i$-good,
there are at least $2^{-16}\dD^4 t$ paths $P_j$ which are $i$-absorbing of $(v,u)$.

Relabel indices so that $I=[s]$ where $s:=(1-2^{-16}\dD^4)t$.
It remains to connect the $P_1,\ldots,P_s$ into an absorbing cycle $C$.
Recall that no colour in $\ov{\ms{U}}$ appears on any $P_i$.
We first connect $P_1$ and $P_2$ into a single transversal path $P_1v^1_4xzyv^2_1P_2$.
For this, we first choose distinct colours $c_1,c_2,c_3,c_4 \in \ov{\ms{U}}$
where $c_1c_2 \in E(H)$.
There are at least $|\ov{\ms{U}}| \geq \dD n/2$
choices for $c_1$, and given this at least  $\dD(H[\ov{\ms{U}}]) \geq \dD n/2$ for $c_2$,
and at least $|\ov{\ms{U}}|-3 > \dD n/3$ choices for each of $c_3,c_4$.
Next we choose unused vertices $x,y$ such that $x \in N_{G_{c_3}}(v^1_4)$ is $G_{c_1}$-good
and $y \in N_{G_{c_4}}(v^2_1)$ is $G_{c_2}$-good.
There are at least $(\frac{1}{2}-\mu)n - 4e(M)-2\eps n$ choices for each of these.
Finally, we choose $z \in N_{G_{c_1}}(x) \cap N_{G_{c_2}}(y)$.
We claim that there are at least $\dD n/5$ choices for $z$.

To prove the claim, since $c_1c_2 \in E(H) \subseteq E(C^{\eps,\dD}_{\bm{G}})$, both $G_{c_1}$ and $G_{c_2}$ are $\eps$-extremal,
so $x \in A_{c_1} \cup B_{c_1}$ and $y \in A_{c_2} \cup B_{c_2}$,
and $G_{c_1}$ and $G_{c_2}$ are $\dD$-crossing.
By Observation~\ref{obs:cross}, we have $|X_{c_1} \cap Y_{c_2}| \geq \dD n/4$
whenever $X,Y \in \{A,B\}$.
By the definition of characteristic partition,
there are $Z,W \in \{A,B\}$ such that $d_{G_{c_1}}(x,Z_{c_1}) \geq |Z_{c_1}|-\eps n$
and $d_{G_{c_2}}(y,W_{c_2}) \geq |W_{c_2}|-\eps n$.
Thus $|N_{G_{c_1}}(x) \cap N_{G_{c_2}}(y)| \geq |N_{G_{c_1}}(x,Z_{c_1}) \cap N_{G_{c_2}}(y,W_{c_2})| \geq |Z_{c_1} \cap W_{c_2}| - 2\eps n \geq \dD n/5$.
This completes the proof of the claim.

Thus there are at least $\dD n/5$ choices for each of $c_1,c_2,c_3,c_4,x,y,z$ given any previous choices.
So we can obtain the desired $P_1xyzP_2$.
We can further repeat the process for each pair $(P_j,P_{j+1})$ for $j \in [s]$ where $P_{s+1} := P_1$.
Each time we use four unused colours in $\ov{\ms{U}}$ and three unused vertices, so for each choice of colour at most $4s \leq 4t \leq 8\kK n < \dD n/10$ colours are forbidden and $3s < \dD n/10$ vertices are forbidden.
Thus we obtain a cycle $C$ of length $7s$ which contains every $P_j$ with $j \in [s]$ as a segment.
By construction, for every colour $i \in \ov{\ms{U}}$ and any $G_i$-good vertex $v$,
for every vertex $u \in \ov{T}_i$, there are at least $2^{-16}\dD^4 t \geq 2^{-17}\kK\dD^5n$ disjoint $i$-absorbing paths of $(v,u)$ inside $C$.
Moreover, since at most $2\eps n$ vertices are $G_i$-bad, the number of vertices in $\ov{T}_i$ which are $G_i$-good is at least $|\ov{T}_i| - 2\eps n \geq (1-\sqrt{\eps}/2)n-2\eps n \geq (1-\sqrt{\eps})n$.
Thus $C$ is an absorbing cycle with parameters
$(\dD/3,\sqrt{\eps},14\kK,2^{-17}\kK\dD^5)$
and hence with $(\dD/3,\sqrt{\eps},\lL,\lL^2)$.
\end{proof}

\section{The stable case}\label{sec:stable}

In this section we combine the results of the previous sections
and use the regularity-blow-up method to find a transversal Hamilton cycle when $\bm{G}$ is stable.
First, we show that the reduced graph inherits stability.

\begin{lemma}\label{lm:reducestable}
Suppose that $0<1/n \ll 1/L_0 \ll \eps_0 \ll d\ll \mu,\aA \ll \gG,\eps \ll \dD<1$. 
Let $\bm{G}=(G_1,\ldots,G_n)$ be a graph collection on a common vertex set $V$ of size $n$.
Let $\bm{R}=\bm{R}(\eps_0,1,L_0)$ be the reduced graph collection of $\bm{G}$.
If $\bm{G}$ is $(\gG,\aA,\eps,\dD)$-stable, then $\bm{R}$ is $(\gG/2,\aA^2,\eps,\dD/2)$-stable.
\end{lemma}
\begin{proof}
We may assume that $1/n \ll 1/n_0$ where $n_0=n_0(\eps,1,L_0)$
is the constant from Lemma~\ref{lm:weakregcol} (the regularity lemma for graph collections).
Write $[L]$ for the common vertex set of $\bm{R}$
and $[M]$ for the set of colour clusters where $L_0 \leq L,M \leq n_0$.
Thus there is a partition $V_0,\ldots,V_L$ of $V$ and $\Cols_0,\ldots,\Cols_M$ of $[n]$
and a graph collection $\bm{G}'$ satisfying (i)--(v) of Lemma~\ref{lm:weakregcol}
and for $(\{h,i\},j) \in \binom{[L]}{2} \times [M]$, we have $hi \in R_j$ 
whenever $\bm{G}'_{hi,j}$ is $(\eps_0,d)$-regular,
where $\bm{G}'_{hi,j} := (G'_c[V_h,V_i]: c \in \Cols_j)$.
We have $Lm \leq n \leq Mm+\eps_0 n$ and $Mm \leq n \leq Lm+\eps_0 n$
so $|L-M| \leq \eps_0 n/m \leq \eps_0 L/(1-\eps_0)$.
Thus we may assume that $M=L$ at the expense of assuming the slightly worse bound $|V_0|+|\Cols_0|\leq 3\eps_0 n$.
Given $X \subseteq V(R)=[L]$, we write $\wh{X} := \bigcup_{j \in X}V_j$.
So $|\wh{X}|=m|X|$.

\begin{claim}\label{cl:allcolour}
Let $\aA' \gg d$.
Suppose that $i \in [L]$ and there are sets $A,B \subseteq V(R)=[L]$ of size at least $(\frac{1}{2}-\aA')L$ such that $e_{R_i}(A,B)\leq \aA' L^2$. 
Then, for all but at most $2 \sqrt{\aA'} m$ colours $c \in \Cols_i$, we have $e_{G_c}(\wh{A},\wh{B})\leq \sqrt{\aA'} n^2$.
\end{claim}
\begin{proof}[Proof of Claim]
Let $t$ be the number of colours $c \in \Cols_i$ such that $e_{G_c}(\wh{A},\wh{B})> \sqrt{\wt{\aA}} n^2$.
Since $e(G_c)-e(G_c') \leq (3d+\eps_0)n^2$ for all $c \in [n]$, we have
\begin{align*}
\sqrt{\aA'} n^2 t &\leq \sum_{c \in \Cols_i}e_{G_c}(\wh{A},\wh{B}) 
\leq \sum_{c \in \Cols_i}(e_{G_c'}(\wh{A},\wh{B}) + (3d+\eps_0)n^2)
\leq |\Cols_i|e_{R_i}(A,B)m^2 + |\Cols_i|(3d+\eps_0)n^2\\
&\leq \aA' L^2m^3+4dmn^2 \leq 2\aA' m n^2.
\end{align*}
This implies $t\leq 2\sqrt{\aA'} m$.
\end{proof}

\medskip
\noindent
\emph{Case 1}: $\bm{G}$ is $(\gG,\aA)$-strongly stable. 
We claim that $\bm{R}$ is $(\gG/2, \aA^2)$-strongly stable.
Suppose for a contradiction that there exists a subset $I \subseteq [L]$ with $|I|\geq (1-\gG/2)L$ such that $R_i$ is not $\aA^2$-nice for any $i\in I$. 
The claim applied with $\aA' = \aA^2$ implies that there are at least $(1-2\aA)m$ colours $c\in \Cols_i$ such that $G_i$ is 
not $\aA$-nice,
since $|\wh{A}|,|\wh{B}| \geq (\frac{1}{2}-\aA^2)Lm > (\frac{1}{2}-\aA)n$.
Thus the number of colours $c \in [n]$ for which $G_c$ is not $\aA$-nice is at least
$(1-2\aA)m(1-\gG/2)L > (1-\gG)n$ since $\aA \ll \gG$.
This contradicts the $(\gG,\aA)$-strong stability of $\bm{G}$.

\medskip
\noindent
\emph{Case 2}: $\bm{G}$ is $(\eps,\dD)$-weakly stable.
Suppose that $\bm{R}$ is not $(\gG/2, \aA^2)$-strongly stable.
It suffices to show that $\bm{R}$ is $(\eps, \dD/2)$-weakly stable. 
There is $I \subseteq [L]$ with $|I| \geq (1-\gG/2)L$ such that for all $i \in I$, 
$R_i$ is not $\aA^2$-nice.
Thus $R_i$ is $\aA^{2/3}$-extremal, and hence $\eps$-extremal.
Lemma~\ref{lm:inherit} implies that $d_{R_i}(j) \geq (\frac{1}{2}-2\mu)L$
for all but at most $d^{1/4}L$ vertices $j \in [L]$.
Lemma~\ref{lm:char} now implies that
$R_i$ has a characteristic partition $(A_i',B_i',C_i')$.
Furthermore, there are $Z,W \in \{A,B\}$ such that
$e_{R_i}(Z_i',W_i') \leq \aA^{2/3}L^2$ (and 
$Z \neq W$
if $R_i$ is $(\aA^{2/3},\ECone)$-extremal,
while $Z=W$ if $R_i$ is $(\aA^{2/3},\ECtwo)$-extremal). 
We have $|\wh{W'_i}| = |\wh{Z'_i}|=|Z_i'|m=(\frac{1}{2}-\aA^{2/3})Lm \geq (\frac{1}{2}-2\aA^{2/3})n$.
By Claim \ref{cl:allcolour} applied with $(W_i',Z_i')$ and $\aA' := \aA^{1/3}$, there is $\ms{B}_i \subseteq \Cols_i$ 
with $|\Cols_i \sm \ms{B}_i| \leq 2\aA^{1/3}m$ such that for all $c \in \ms{B}_i$, we have $e_{G_c}(\wh{W'_i},\wh{Z_i'})\leq \aA^{1/3} n^2$. It follows that every such $G_c$ is $2\aA^{1/3}$-extremal.
Thus, recalling that $(A_c,B_c,C_c)$ is the characteristic partition of $G_c$,
there are $Z,Y,W \in \{A,B\}$ with $\{Z,Y\}=\{A,B\}$ such that $e_{G_c}(Z_c,W_c) \leq \aA^{1/4}n^2$ and $e_{G_c}(Y_c,W_c) \geq (\frac{1}{4}-\aA^{1/4})n^2$. 
But then we must have $|\wh{Z'_i} \sd Z_c|,|\wh{Y'_i} \sd Y_c|< \eps n$ 
or $|\wh{Z'_i} \sd Y_c|,|\wh{Y'_i} \sd Z_c|< \eps n$.
So either $|\wh{A'_i} \sd A_c|,|\wh{B'_i} \sd B_c| < \eps n$ or $|\wh{A'_i} \sd B_c|,|\wh{B'_i} \sd A_c| < \eps n$.
That is, for all $c\in \ms{B}_i$, the characteristic partition of $G_c$ is almost the same as the union of clusters corresponding to the characteristic partition of $R_i$.

Suppose now that $c,c' \in [n]$ are such that $G_c$ and $G_{c'}$ are $\eps$-extremal and $\dD$-crossing
and $c \in \ms{B}_i$ and $c' \in \ms{B}_j$ for some $i,j \in [L]$.
We must have $i \neq j$. Then, if $|\wh{A'_i} \sd A_c|,|\wh{A'_j} \sd A_{c'}| < \eps n$, we have 
$|A_i' \sd A_j'|m = |\wh{A'_i} \sd \wh{A'_j}| \geq |A_c \sd A_{c'}| - |\wh{A'_i} \sd A_c| - |\wh{A'_j} \sd A_{c'}| \geq \dD n - 2\eps n \geq \dD n/2 \geq \dD mL/2$.
The other cases are almost identical.
Thus $R_i$ and $R_j$ are $\dD/2$-crossing.
Since $\bm{G}$ is $(\eps,\dD)$-weakly stable, we have $e(C^{\eps,\dD}_{\bm{G}}) \geq \dD n^2$.
The number of pairs $i,j \in [L]$ with $c,c'$ as above is therefore at least 
$\dD n^2/m^2 \geq \dD L^2/2$.
Thus $e(C^{\eps,\dD/2}_{\bm{R}}) \geq \dD L^2/2$,
and hence $\bm{R}$ is $(\eps,\dD/2)$-weakly stable.
\end{proof}

\begin{lemma}\label{lm:rhcstable}
Let $0<1/n\ll \mu \ll \aA \ll \gG,\eps\ll \dD \ll 1$. Suppose that $\bm{G}=(G_1,\ldots,G_n)$ is a graph collection on a common vertex set $V$ of size $n$ and $\dD(\bm{G})\geq (\frac{1}{2}-\mu)n$. If $\bm{G}$ is $(\gG, \aA,\eps, \dD)$-stable, then $\bm{G}$ contains a transversal Hamilton cycle.
\end{lemma}

\begin{proof}
Choose additional parameters $n_0,L_0,\eps_0,d,\theta,\bB,\lL$ 
where $n_0=n_0(\eps,1,L_0)$ is obtained from Lemma~\ref{lm:weakregcol} and
so that
$$
0<1/n \ll 1/n_0 \ll 1/L_0 \ll \eps_0 \ll d\ll \mu \ll \theta\ll \bB \ll \lL \ll 
\aA \ll \gG,\eps\ll \dD \ll 1,
$$
where the previous lemmas in this section hold with suitable parameters.
For any colour $c$ and two vertices $(x,y) \in V^2$, we say the triple $(c,x,y)$ is \emph{absorbable} if 
there are at least $\lL^2 n$ disjoint $c$-absorbing paths of $(x,y)$ inside $C$. Similarly, we say the pair $(c,x)$ is \emph{absorbable} if 
there are at least $\lL^2 n$ disjoint $c$-absorbing paths of $(x,x)$ inside $C$. 
Let $c,c' \in [n]$ be two colours and $x,y \in V$ be two vertices.
We say $(c,c',x,y)$ is \emph{totally absorbable} if $(c,x)$, $(c',y)$ and $(c,x,y)$ are all absorbable.
By Lemma \ref{lm:absorbcycle} and Lemma \ref{lm:absorbcycle2}, $\bm{G}$ has an absorbing cycle $C$ with parameters $(\dD/3,\sqrt{\eps},\lL, \lL^2)$. 
By definition, $C$ has length at most $\lL n$, and there exists a colour set $\Cols \subseteq [n]$ with $\Cols \cap \col(C) = \es$ of size at least $\dD n/3$ such that
\begin{itemize}
\item[(i)] given any colour $c\in \Cols$ and any $G_c$-good vertex $v$, 
the triple $(c,v,u)$ is absorbable for all but at most $\sqrt{\eps}n$ vertices $u$.
\item[(ii)] given any colour $c\in \Cols$, for all but at most $\sqrt{\eps} n$ $G_c$-good vertices $v$, 
the pair $(c,v)$ is absorbable.
\end{itemize}

\begin{claim}\label{cl:stableabs}
There is $(2-2^{-10})\bB n \leq r \leq 2\bB n$ and for each $i \in [r]$, disjoint vertex pairs $(v_i,v_i') \in (V)_2$ 
and disjoint colour pairs $(c_i,c_i') \in (\Cols)_2$ 
such that the family $\mc{Q} := \{(c_i,c_i',v_i,v_i'): i \in [r]\}$
has the following properties:
\begin{itemize}
\item[(a)] $(c_i,c_i',v_i,v_i')$ is totally absorbable for all $i \in [r]$;
\item[(b)] for every pair $(u_1,u_2) \in V^2$ and $c \in [n]$, there are at least $2^{-9}\bB n$ 
values $i \in [r]$ such that 
$c \in L(u_1v_i)$ and $c_i' \in L(u_2v_i')$.
\end{itemize}
\end{claim}

\begin{proof}[Proof of Claim]
For every $(b_1,b_2) \in (\Cols)_2$, 
every $(u_1,u_2)\in V^2$ and $c \in \Cols$, 
Let $S(b_1,b_2,u_1,u_2,c)$ be the multiset of pairs $(v_1,v_2) \in (V)_2$ such that $c \in L(u_1v_1)$, $b_2 \in L(u_2v_2)$ and $(b_1,b_2,v_1,v_2)$ is totally absorbable.
We will show that $|S(b_1,b_2,u_1,u_2,c)| \geq 2^{-4}n^2$.

For this, we first count the number of choices for $v_1$. Note that we have $d_{G_c}(u_1)\geq (\frac{1}{2}-\mu)n$ and the number of vertices $v_1$ such that $v_1$ is not $G_{b_1}$-good or $(b_1,v_1)$ is not absorbable is at most $2\eps n+\sqrt{\eps}n$. Thus we have at least $n/4$ choices of $v_1$ such that $v_1\in N_{G_c}(u_1)$, $v_1$ is $G_{b_1}$-good and $(b_1,v_1)$ is absorbable. 
Now we fix $v_1$. 
Let $N_2 := \{x \in N_{G_{b_2}}(u_2): x \text{ is }G_{b_2}\text{-good}\}$. 
So $|N_2|\geq (\frac{1}{2}-\mu)n-2\eps n$. By (i), for all but at most $\sqrt{\eps} n$ vertices $v_2\in V$, the pair $(v_1,v_2)$ has at least $\lL^2 n$ $b_1$-absorbing paths inside $C$. Thus, we can delete at most $\sqrt{\eps} n+1$ vertices from $N_2$ and obtain a set $N_2'$ such that $|N_2'|\geq n/4$, and for each $v_2\in N_2'$, $v_2 \neq v_1$ and $(b_1,b_2,v_1,v_2)$ is totally absorbable. Therefore, the total number of choices for $(v_1,v_2)$ is at least $2^{-4}n^2$.

Let $\{c_i,c_i' : i \in [2\bB n]\} \subseteq \Cols$ be a collection of distinct colours.
Set $\bm{H}:=(\mc{F}_i := \{(v_1,v_2) \in (V)_2: (c_i,c_i',v_1,v_2)\text{ is totally absorbable}\} : i \in[2\bB n])$ and define the collection $\bm{Z}:=\{S(u_1,u_2,c) := \bigcup_{i \in [2\bB n]} S(c_i,c_i',u_1,u_2,c) : (u_1,u_2)\in V^2, c\in [n]\}$ of multi-$2$-graphs.
For every $S = S(u_1,u_2,c) \in \bm{Z}$ and every index $i \in [2\bB n]$, we have
$|E(S) \cap E(\mc{F}_i)| \geq 2^{-4}n^2$.
Lemma~\ref{lm:key} applied with $2\bB n,2^{-4}$ playing the roles of $t,\eps$,
implies that there is a transversal matching $M$ in $\bm{H}$ of size at least $(2- 2^{-10})\bB n$ 
(and at most $2\bB n$) 
and $|E(S) \cap E(M)| \geq 2^{-9}\bB n$ for all $S \in \bm{Z}$.
\end{proof}

For the remainder of the proof, we will cover most of the unused vertices and colours by a small number of long paths. This step will use the regularity-blow-up method for transversals.
Finally, we will use the absorbing cycle $C$ and family $\mc{Q}$ to absorb the remaining colours and vertices and connect the long paths.

Let
\begin{gather*}
V_{\abs} := \bigcup_{i \in [r]}\{v_i,v_i'\},\quad
\Cols_{\abs} := \bigcup_{i \in [r]}\{c_i,c_i'\},
\quad
\Cols_{\rm rem}:=[n]\sm (\col(C)\cup \Cols_{\abs}),\\
\quad
U:=V\sm (V(C)\cup V_{\abs})
\quad\text{and}\quad
\bm{J}=(J_i : i \in \Cols_{\rm rem})
\text{ where }J_i:=G_i[U].
\end{gather*} 
So $|[n] \sm \Cols_{\rm rem}|=|V \sm U| \leq \lL n+2\bB n \leq 2\lL n$.

Since $\lL \ll \gG,\aA,\eps,\dD$, it is easy to see that since $\bm{G}$ is $(\gG, \aA, \eps, \dD)$-stable, $\bm{J}$ is $(\gG/2, \aA/2, 2\eps,\dD/2)$-stable, and $\dD(J_i)\geq (\frac{1}{2}-\mu-2\lL)n \geq (\frac{1}{2}-3\lL)n$ for $i\in \Cols_{\rm rem}$. 
Apply Lemma~\ref{lm:weakregcol}
(the regularity lemma for graph collections) to $\bm{J}$ with parameters $(\eps_0,1,d,L_0)$.
Let $\bm{R}$ be the reduced graph of $\bm{J}$.
Write $[L]$ for the common vertex set of $\bm{R}$ where $L_0 \leq L \leq n_0$,
and $[M]$ for the set of colour clusters.
Thus there is a partition $V_0,\ldots,V_L$ of $V$ and $\Cols_0,\ldots,\Cols_M$ of $[n]$
and a graph collection $\bm{J}'$ satisfying (i)--(v) of Lemma~\ref{lm:weakregcol}.
Therefore, for $(\{h,i\},j) \in \binom{[L]}{2} \times [M]$, we have $hi \in R_j$ 
whenever $\bm{J}'_{hi,j}$ is $(\eps_0,d)$-regular,
where $\bm{J}'_{hi,j} := (J'_c[V_h,V_i]: c \in \Cols_j)$.
We have $Lm \leq n \leq Mm+\eps_0 n$ and $Mm \leq n \leq Lm+\eps_0 n$
so $|L-M| \leq \eps_0 n/m \leq \eps_0 L/(1-\eps_0)$.
Thus we may assume that $M=L$ at the expense of assuming the slightly worse bound $|V_0|+|\Cols_0|\leq 3\eps_0 n$. 
Lemma~\ref{lm:inherit} implies that for each vertex $i \in [L]$, there are at least $(1-d^{1/4})L$ colours $j \in [L]$ such that $d_{R_j}(i)\geq (\frac{1}{2}-4\lL)L$. By Lemma \ref{lm:reducestable}, $\bm{R}$ is $(\gG/4,\aA^2/4,2\eps,\dD/4)$-stable. 
Lemma~\ref{lm:2matching} implies that
there exist two edge-disjoint transversal matchings $M_1$ and $M_2$ in $\bm{R}$ such that $e(M_\ell)\geq (\frac{1}{2}-\theta)L$ for $\ell=1,2$ and $\col(M_1)\cap \col(M_2)=\es$.

Now for each vertex $i \in [L]$, let $V_i = V_i^1 \cup V_i^2$ be an arbitrary equipartition.
Lemma~\ref{lm:slice}(i) implies that whenever $hi \in E(R_j)$ where $j \in [L]$,
we have that $\bm{J}^\ell_{hi,j} = (J_c'[V_h^\ell,V_i^\ell]:c \in \Cols_j)$ is $(2\eps_0,d/2)$-regular for both $\ell=1,2$.
By Lemma~\ref{lm:standard}, for each $\ell =1,2$ and $hi \in E(M_\ell)$, we can remove at most $2\eps_0 m$ vertices in $V_h^\ell$
and at most $2\eps_0 m$ colours $c \in \Cols_j$, where $j$ is the colour of $hi$ in $M_\ell$,
so that the remaining graph collection is $(4\eps_0,d/4)$-superregular.

Apply Theorem~\ref{th:blowup} (the transversal blow-up lemma) to obtain a transversal path $P_{hi,j}$ with $2\min\{|V_h^\ell|,|V_i^\ell|\}$ vertices using colours from $\Cols_j$.
(Note that we could have avoided using the blow-up lemma and instead used a tool for embedding an almost spanning structure inside a regular pair (as opposed to superregular), but it was convenient to follow the above approach.)
The collection of $P_{hi,j}$ over $hi \in E(M_1) \cup E(M_2)$ (where $hi$ has colour $j$ in its matching)
is vertex-disjoint, since the subcluster $V^\ell_h$ is used $d_{M_\ell}(h) \leq 1$ times,
 and transversal since $M_1 \cup M_2$ is transversal.
The number of vertices not in any $P_{hi,j}$ is at most $4\eps_0 mL + 3\eps_0 n + 2(L-e(M_1)-e(M_2))m \leq 3\theta n$.
Thus the number of colours not in any $P_{hi,j}$ is at most $3\theta n+e(M_1)+e(M_2) \leq 4\theta n$.
We consider each vertex in $U$ but not in any $P_{hi,j}$ to be a path (of length $0$).

Now relabel so that the paths are $P_1,\ldots,P_s$, so $s\leq e(M_1)+e(M_2)+3\theta n \leq 4\theta n$,
and let $x_i,y_i$ be the startvertex and endvertex of $P_i$ for each $i \in [s]$
(so $x_i=y_i$ if $P_i$ has length $0$).
The paths $P_1,\ldots,P_s$ are transversal and pairwise vertex-disjoint, and cover $U$.
Thus the number of colours in $\Cols_{\rm rem}$ which are not used on any $P_i$ is precisely $s$.

Do the following for each $i=1,\ldots,s$ in turn.
Let $a_i$ be an arbitrary unused colour, in $\Cols_{\rm rem}$.
Choose an unused $4$-tuple $Q_i=(c_{j_i},c_{j_i}',v_{j_i},v_{j_i}') \in \mc{Q}$ where $j_i \in [r]$,
$a_i \in L(x_iv_{j_i})$ and $c_{j_i}' \in L(y_iv_{j_i}')$. 
This is possible since Claim~\ref{cl:stableabs} implies there are at least $2^{-9}\bB n$ choices for $Q_i$,
of which at most $s \leq 4\theta n$ have been used.
Now, $(c_{j_i},c_{j_i}',v_{j_i},v_{j_i}')$ is totally absorbable, so 
there are at least $\lL^2 n$ disjoint $c_{j_i}$-absorbing paths of $(v_{j_i},v_{j_i}')$ inside $C$.
So we can choose one of them, $S_i:= x^i_1x^i_2x^i_3x^i_4$, whose vertices have not been previously chosen since $s/n \ll \lL$, and whose colours are, in order, $b^i_1,b^i_2,b^i_3$.

At the end of this process, there remains $I \subseteq [r]$ such that the $(c_j,c_j',v_j,v_j') \in \mc{Q}$ with $j \in I$ are precisely the $4$-tuples which were not chosen to be some $Q_i$. For each one, 
there are at least $\lL^2 n$ disjoint $c_j$-absorbing paths of $(v_j,v_j)$ 
and disjoint $c_j'$-absorbing paths of $(v_j',v_j')$ inside $C$.
So, since $\bB \ll \lL$, we can choose one such path $T_i,T_i'$ for each of $(v_j,c_j)$ and $(v_j',c_j')$, which are vertex-disjoint and whose vertices have not previously been chosen.

At the end of this process, we have a collection $\{S_i:i \in [s]\},\{T_i:i \in I\}, \{T_i':i \in I\}$ of vertex-disjoint paths in $C$
where, for each $i \in [s]$, $S_i$ is a $c_{j_i}$-absorbing path of $(v_{j_i},v_{j_i}')$;
for each $i \in I$, $T_i$ is a $c_i$-absorbing path of $(v_i,v_i)$ and $T_i'$ is a $c_i'$-absorbing path of $(v_i',v_i')$.
For each $i \in [s]$, we replace $S_i$ by $x^i_1x^i_2v_{j_i}x_iP_iy_iv_{j_i}'x^i_3x^i_4$
with colours $b^i_1,c_{j_i},a_i$, followed by the colours inherited from $P_i$, 
followed by $c_{j_i}',b^i_2,b^i_3$.
That is, we have replaced $S_i$ by a path with the same endpoints, vertices $V(S_i) \cup V(P_i) \cup \{v_{j_i},v_{j_i}'\}$ and colours $\col(S_i) \cup \col(P_i) \cup \{c_{j_i},c_{j_i}',a_i\}$.
For each $i \in I$, we replace $T_i=y^i_1y^i_2y^i_3y^i_4$ by $y^i_1y^i_2v_iy^i_3y^i_4$
where colours are inherited except $\col(y^i_2v_i)=c_i$ and $\col(v_iy^i_3)=\col(y^i_2y^i_3)$.
We do a similar replacement of $T_i'$, using new vertex $v_i'$ and new colour $c_i'$.
So we have replaced $T_i$ by a path with the same endpoints, vertices $V(T_i) \cup \{v_i\}$
and colours $\col(T_i) \cup \{c_i\}$,
and $T_i'$ by a path with the same endpoints, vertices $V(T_i') \cup \{v_i'\}$
and colours $\col(T_i') \cup \{c_i'\}$.
Thus we have obtained a cycle using vertices $V(C) \cup \bigcup_{i \in [s]}V(P_i) \cup \{v_i,v_i': i \in [r]\}=[n]$ and colours $\col(C) \cup \bigcup_{i \in [s]}\col(P_i) \cup \{c_i,c_i':i \in [r]\} = [n]$
where each colour is used at most once (and hence exactly once).
That is, we have constructed a transversal Hamilton cycle.
\end{proof}

\section{The extremal case}\label{sec:ext}

This section concerns the remaining case when $\bm{G}$ is not stable;
so most graphs in the collection are close to containing one of the two extremal graphs $\ECone$ or $\ECtwo$,
and moreover their characteristic partitions are similar.
Throughout, we assume the following hypothesis:
 \medskip
\begin{enumerate}[label=$(\dagger)$,ref=$(\dagger)$]
\item\label{it:dagger}
Suppose that
$$
0< 1/n \ll \eps \ll \dD \ll \eta \ll 1.
$$
Let $m=(1-3\dD)n$.
Let $\bm{G}=(G_1, \ldots, G_{n})$ be a graph collection on a common vertex set $V$ of size $n$ and $\dD(\bm{G})\geq (\frac{1}{2}-\mu)n$. Suppose that for each $i\in [m]$, $G_i$ is $\eps$-extremal with characteristic partition $(A_i,B_i,C_i)$
and $|A_1 \sd A_i|,|B_1 \sd B_i| \leq \dD n$.
\end{enumerate}

\medskip
\noindent
We also use the following notation:
\begin{align*}
\Cols(H)&:=\{i\in [m]: G_i\text{ is }(\eps,H)\text{-extremal}\}\quad
\text{and}\quad
\bm{G}(H) :=(G_i: i \in \Cols(H))
\quad
\text{for }H \in \{\ECone,\ECtwo\},\\
V_{\rm bad} &:= C_1 \cup V^A_{\rm bad} \cup V^B_{\rm bad}
\quad\text{where}\quad V^Z_{\rm bad} := \{x \in Z_1: x \notin Z_i\text{ for at least }\sqrt{\dD}n\text{ colours }i \in [m]\},\\
\Cols_{\rm bad} &:= [m+1,n].
\end{align*}
Now,
$\sqrt{\dD}n|V^Z_{\rm bad}|
\leq \sum_{i \in [m]}|Z_1 \sm Z_i| \leq m\dD n$
so
$$
|V_{\rm bad}| \leq 2\eps n + 2\sqrt{\dD}m \leq 3\sqrt{\dD}n.
$$
It is a consequence of Lemma~\ref{lm:char} that $\Cols(\ECone) \cap \Cols(\ECtwo)=\es$.
The first result in this section is an application of the transversal blow-up lemma for embedding transversal Hamilton paths inside very dense bipartite graph collections.
We note that this application is more for convenience than necessity.

\begin{lemma}\label{lm:almostcp1}
Suppose that~\ref{it:dagger} holds.
Let $W,Z \in \{A,B\}$.
Let $W^* \subseteq W_1 \sm V_{\rm bad}$
and $Z^* \subseteq Z_1 \sm V_{\rm bad}$
where $|W^*|,|Z^*| \geq \eta n$
and $W^* \cap Z^* =\emptyset$
and $|W^*|-|Z^*| =: t \in \{0,1\}$. 
Let $T:=Z$ if $t=0$ and $T:=W$ if $t=1$.
Let $\Cols \subseteq [n]$ satisfy $|\Cols|=|W^*|+|Z^*|-1$,
where $\Cols \subseteq \Cols(\ECone)$ if $W=Z$ and $\Cols \subseteq \Cols(\ECtwo)$ if $W \neq Z$.
Let $W^- \subseteq W^*$ and $T^+ \subseteq T^*$
with $|W^-|,|T^+| \geq \eta n$.
Then there is a transversal Hamilton path in $\{G_i[W^*,Z^*]: i \in \Cols\}$ starting in $W^-$ and ending in $T^+$.
\end{lemma}

\begin{proof}
Choose a new constant $\xi$ with $\eta \ll \xi \ll 1$.
We will show that $\{G_i[W^*,Z^*]: i \in \Cols\}$ is $(\xi,1-\xi)$-superregular.

Suppose that $W \neq Z$, so $\{W^*,Z^*\}=\{A^*,B^*\}$ and $\Cols \subseteq \Cols(\ECtwo)$.
For every $i \in \Cols$, we have $e_{G_i}(W_i,Z_i) \geq |W_i||Z_i|-\eps n^2$
and hence $e_{G_i}(W^*,Z^*) > (1-\xi)|W^*||Z^*|$.
Now let $x \in W^*$ and $\Cols' \subseteq \Cols$ with $|\Cols'| \geq \xi|\Cols|$
and $Z' \subseteq Z^*$ with $|Z'| \geq \xi|Z^*|$. Then $x \notin V_{\rm bad}$, 
so $x \in W_i$ for all but at most $\sqrt{\dD}n$ colours in $i \in [m]$.
For each such $i$ we have $d_{G_i}(x,Z_i) \geq (\frac{1}{2}-2\eps)n = |Z_i|-\eps n$
by Lemma~\ref{lm:char}.
Thus $d_{G_i}(x,Z') \geq |Z'|-\eps n-|Z_1 \sd Z_i| \geq |Z'|-2\dD n$.
Thus $\sum_{i \in \Cols'}d_{G_i}(x,Z') \geq (|\Cols'|-2\sqrt{\dD}n)(|Z'|-2\dD n) > (1-\xi)|\Cols'||Z'|$.
An analogous statement holds for vertices in $Z^*$.
Thus $\sum_{i \in \Cols}d_{G_i}(x,Z^*) \geq (1-\xi)|\Cols||Z^*|$ for all $x \in W^*$
and $\sum_{i \in \Cols}d_{G_i}(y,W^*) \geq (1-\xi)|\Cols||W^*|$ for all $y \in Z^*$
and $(1-\xi)|\Cols'||W'||Z'| \leq \sum_{i \in \Cols'}e_{G_i}(W',Z') \leq |\Cols'||W'||Z'|$.
Thus $(G_i[W^*,Z^*]: i \in \Cols)$ is $(\xi,1-\xi)$-superregular.

Suppose instead that $W = Z$, so $\Cols \subseteq \Cols(\ECone)$.
A very similar argument shows that, since every $G_i[Z^*]$ and hence $G_i[W^*,Z^*]$ is almost complete,
the same conclusion holds here.

Let $s := |W^*|+|Z^*|$.
Since $|W^*|-|Z^*|=t\in\{0,1\}$,
there is an $s$-vertex path $P_s$ which is bipartite with parts $A_W,A_Z$ of size $|W^*|,|Z^*|$ respectively,
with $|\Cols|=|W^*|+|Z^*|-1$ edges and whose first vertex lies in $A_W$ and whose last lies in $A_T$. 
Thus we can apply Theorem~\ref{th:blowup} (the transversal blow-up lemma) with target sets $W^-,T^+$
for the first and last vertex of $P_s$ respectively
to obtain the required transversal Hamilton path.
\end{proof}

The next lemma shows that whenever there are many $(\eps,\ECtwo)$-extremal graphs,
we can find a short transversal path which covers bad vertices and colours.

\begin{lemma}\label{lm:cover1}
Suppose that~\ref{it:dagger} holds,
and
$
|\Cols(\ECtwo)|\geq \eta n
$. 
Given any $\ms{F} \subseteq [n]$ with $|\ms{F}| \leq 1$, 
there is a transversal path $P$ in $\bm{G}$ with endpoints $x,y$ 
such that the following holds:
 \begin{enumerate}[(i)]
\item $V_{\rm bad} \subseteq V(P)$ and $\Cols_{\rm bad} \sm \ms{F} \subseteq \col(P)$;
\item $|V(P)|\leq 19\sqrt{\dD} n$;
\item $x,y \notin V_{\rm bad}$ and there are distinct $c,c' \in \Cols(\ECtwo) \sm \col(P)$ such that $x\in A_c \cap A_1$ and $y\in B_{c'} \cap B_1$;
\item $\ms{F} \cap (\col(P) \cup \{c,c'\})=\emptyset$.
\end{enumerate}
\end{lemma}

\begin{proof}
Since $|V_{\rm bad}| \leq 3\sqrt{\dD}n$ and $|\Cols_{\rm bad}| \leq 3\dD n$,
by adding vertices to $V_{\rm bad}$ if necessary we may assume that
$|\Cols_{\rm bad}| < r:=|V_{\rm bad}| \leq 3\sqrt{\dD}n$.
Let $A := A_1 \sm V_{\rm bad}$ and $B := B_1 \sm V_{\rm bad}$.

We will find a transversal family $\mc{P}=\{y_ix_iy_i': i \in [r]\}$ of vertex-disjoint $3$-vertex paths, with colour set $\{c_i,c_i': i \in [r]\} \supseteq \Cols_{\rm bad}$ where $V_{\rm bad}=\{x_1,\ldots,x_r\}$ and
$\{y_i',y_i':i \in [r]\} \subseteq A \cup B$.
For this, let $x\in V_{\rm bad}$ and let $c,c'\in [m]$ be distinct colours.
Since $d_{G_{c}}(x)\geq (\frac{1}{2}-\mu)n$,
there are $Z,Z' \in \{A,B\}$ such that $|N_{G_{c}}(x)\cap Z|\geq n/8$
and $N_{G_{c'}}(x) \cap Z'| \geq n/8$. 
Choose $y\in N_{G_{c}}(x)\cap Z$ and $y'\in N_{G_{c'}}(x) \cap Z'$. 
This gives the path $yxy'$ using the given colours $c,c'$.
We find such paths for every $x \in V_{\rm bad}$, each time using unused vertices,
which is possible since there are at least $n/8$ choices for each vertex and only $2r\leq 6\sqrt{\dD}n$ are used in total,
and using all of the $3\dD n$ colours of $\Cols_{\rm bad}$
among the $2r$ colours used in total.

Next, we connect the paths of $\mc{P}$ into a single short transversal path $P$. 
We start with two arbitrary paths $P_1,P_2$ in $\mc{P}$,
with endpoints $y_1,y_1'$ and $y_2,y_2'$ respectively.
Note that each of these endpoints is not in $V_{\rm bad}$ and is therefore $G_i$-good for at least $|\Cols(\ECtwo)|-\sqrt{\dD}n-2r\geq \eta n/2$ unused colours $i\in \Cols(\ECtwo)$.
Thus we can choose distinct unused colours $j_1,j_2\in \Cols(\ECtwo)$ such that $y_1'$ is $G_{j_1}$-good and $y_2$ is $G_{j_2}$-good.
Without loss of generality, we have the following two cases:

\medskip
\noindent
\emph{Case 1}: $y_1'\in A_{G_{j_1}}$ and $y_2\in A_{G_{j_2}}$.
Choose a common neighbour $y\in N_{G_{j_1}}(y_1')\cap N_{G_{j_2}}(y_2)\cap B$ that avoids $P_1,P_2$. There are at least $n/3$ choices for $y$
since $y_1'$ is missing at most $\eps n$ neighbours in $B_{j_1}$ in the graph $G_{j_1}$
and $|B_{j_1} \sm B| \leq |V_{\rm bad}|+|B_{j_1} \sd B_1| \leq 4\sqrt{\dD}n$,
and similarly for $y_2$ and $G_{j_2}$.
Let $P_{12}:=P_1y_1'yy_2P_2$ with colour set $\col(P)=\col(P_1)\cup  \{j_1,j_2\} \cup \col(P_2)$. 

\medskip
\noindent
\emph{Case 2}: $y_1'\in A_{G_{j_1}}$ and $y_2\in B_{G_{j_2}}$.
Choose an unused colour $j_3\in \Cols(\ECtwo)$, for which there are at least $\eta n/2$ choices.
Choose a vertex $y\in N_{G_{j_1}}(y_1')\cap B_{G_{j_3}}$ that avoids $P_1,P_2$;
there are at least $n/3$ choices.
Choose a common neighbour $z\in N_{G_{j_3}}(y)\cap N_{G_{j_2}}(y_2)\cap A$ that is distinct from $y_1'$ and avoids $P_1,P_2$;
there are at least $n/3$ choices. 
Let $P_{12} :=P_1y_1'yzy_2P_2$ with colour set $\col(P)=\col(P_1)\cup \{j_1,j_2,j_3\} \cup \col(P_2)$.

\medskip
By considering $P_{12},P_3$ and so on, we can find a transversal path $P_{12\ldots r}$ that includes every path in $\mc{P}$.
For this, we need to argue that we can always choose unused vertices and colours;
this is indeed true since in the first step there are at least $n/3$ choices for any vertex and at least $\eta n/2$ choices for any colour, and at most $3r < \eta n/4$ vertices and colours are used in total.
 We have $|V(P_{12\ldots r})| \leq 6r$.

Using very similar arguments, we can extend the path by at most one vertex and colour
in $\Cols(\ECtwo)$ so that its endpoints are in $A$ and $B$.
Since they do not lie in $V_{\rm bad}$, there are many choices of unused colour for $c,c'$
to satisfy~(iii).
The final path has length at most $6r+1 \leq 19\sqrt{\dD}n$.
To achieve~(iv), we simply remove $\ms{F}$ from $\Cols_{\rm bad}$ and choose not to use $\ms{F}$ at any step,
which does not affect any of the above estimates.
\end{proof}

The next lemma shows that some additional assumptions guarantee a short transversal path
that not only covers bad vertices and colours, but also balances $A_1$ and $B_1$.
These assumptions are that almost every graph is $(\eps,\ECtwo)$-extremal,
and moreover both parts $A_1,B_1$ contain many internal edges from these graphs.

\begin{lemma}\label{lm:cover2}
Suppose that~\ref{it:dagger} holds,
and
$$
|\Cols(\ECone)| \leq \eta n 
\quad\text{and}\quad
e_{\bm{G}(\ECtwo)}(A_1)\geq 30\eta n^3
\quad\text{and}\quad
e_{\bm{G}(\ECtwo)}(B_1)\geq 30\eta n^3.
$$
Then $\bm{G}$ contains a transversal path $P$ with endpoints $x,y$ 
such that the following holds:
 \begin{enumerate}[(i)]
\item $V_{\rm bad} \subseteq V(P)$ and $\Cols_{\rm bad}\cup \Cols(\ECone)\subseteq \col(P)$;
\item $|V(P)|\leq 4\eta n$;
\item $|A_1\sm V(P)|=|B_1\sm V(P)|$;
\item $x,y \notin V_{\rm bad}$ and there are distinct $c,c' \in \Cols(\ECtwo) \sm \col(P)$ such that $x\in A_c \cap A_1$ and $y\in B_{c'} \cap B_1$.
\end{enumerate}
\end{lemma}

\begin{proof}
We have that $|\Cols(\ECtwo)| = m -|\Cols(\ECone)| \geq (1-2\eta)n$.
Thus Lemma~\ref{lm:cover1} applied with $\ms{F} := \emptyset$
implies that $\bm{G}$ contains a transversal path $\wt{P}$
with endpoints $\wt{x},\wt{y}$ such that
$V_{\rm bad} \subseteq V(\wt{P})$,
$\Cols_{\rm bad} \subseteq \col(\wt{P})$,
$|V(\wt{P})| \leq 19\sqrt{\dD}n$,
and $\wt{x},\wt{y} \notin V_{\rm bad}$ and 
$\wt{x} \in A_1$ and $\wt{y} \in B_1 \cap B_c$ for some $c \in \Cols(\ECtwo) \sm \col(\wt{P})$.

Without loss of generality, suppose that $|A_1 \sm V(\wt{P})| - |B_1 \sm V(\wt{P})| := \wt{t} \geq 0$.
Next we greedily extend $\wt{P}$ to a path whose colour set contains $\Cols(\ECone)$.
Write $\{a_1,\ldots,a_r\}$ for the collection of unused colours in $\Cols(\ECone)$ and
let $a_0,a_{r+1}$ be distinct unused colours in $\Cols(\ECtwo)$, for which $\wt{x} \in A_{a_0}$, of which there are at least $\eta n/2$ choices because $\wt{x} \notin V_{\rm bad}$.
Choose an unused vertex $y_0 \in N_{G_{a_0}}(\wt{x},B_1 \cap B_{a_1})$.
There are at least $n/3$ choices since $|B_{a_0} \sd (B_1 \cap B_{a_1})| \leq 2\dD n$ by~\ref{it:dagger} and $d_{G_{a_0}}(\wt{x},B_{a_0}) \geq (\frac{1}{2}-2\eps)n = |B_{a_0}|-\eps n$ because $a_0 \in \Cols(\ECtwo)$ and $\wt{x} \in A_{a_0}$.
Now, given $i \in [r]$ and $y_{i-1} \in B_{a_i}$, we can
choose an unused vertex $y_i \in N_{G_{a_i}}(y,B_1 \cap B_{a_{i+1}})$; again there are at least $n/3$ choices
since $\{a_i,a_{i+1}\} \in [m]$ and~\ref{it:dagger} implies
$|B_{a_i} \sd (B_1 \cap B_{a_{i+1}})| \leq 2\dD n$,
and
$a_i \in \Cols(\ECone)$ implies $d_{G_{a_i}}(y,B_{a_i}) \geq (\frac{1}{2}-2\eps)n=|B_{a_i}|-\eps n$.
Thus we can obtain the transversal path $\wt{x}y_0y_1\ldots y_r$ using colours $a_0,a_1,\ldots,a_r$, and with $y_r \in B_1 \cap B_{a_{r+1}}$.

Let $P_1 = \wt{y}\wt{P}\wt{x}y_0y_1y_1\ldots y_r$ be the final transversal path obtained by concatenation with $\wt{P}$ and
let $A := A_1 \sm V(P_1)$ and $B := B_1 \sm V(P_1)$.
It satisfies (i) and has $|V(P_1)| \leq |V(\wt{P})| + |\Cols(\ECone)|+1 \leq 2\eta n$,
and also $\left| |A|-|B| \right| \leq 2\eta n$.
It has endpoints $\wt{y} \in B_1 \cap B_c$ and $\wt{z} := y_r \in B_b$
where $b := a_{r+1} \in \Cols(\ECtwo)$.
Also, $|A_1 \sm V(P_1)| - |B_1 \sm V(P_1)| := t \geq \wt{t}+1 \geq 1$.
After removing used colours, 
we have $\sum_{i\in \Cols(\ECtwo)}e_{G_i}(A)\geq 20\eta n^3$ and $\sum_{i\in \Cols(\ECtwo)}e_{G_i}(B)\geq 20\eta n^3$ since we consumed at most $2\eta n$ vertices from $A_1\cup B_1$ and at most $2\eta n$ colours from $\Cols(\ECtwo)$ in building $P_1$.
For each vertex pair $uv \in \binom{V}{2}$, define
$$
c(uv):=\{i\in \Cols(\ECtwo) : uv\in E(G_i)\}
\quad\text{and}\quad
D :=\{uv : u,v\in A\text{ and } c(uv)\geq 10\eta n\}. 
$$
Then we have $20\eta n^3\leq |D|n+n^2\cdot 10\eta n$ and thus $|D|\geq 10\eta n^2$. It follows that $D$ contains a subgraph $D'$ with minimum degree at least $10\eta n$.

Let $z \in N_{G_{b}}(\wt{z},V(D'))$ be an unused vertex. Such a $z$ exists since $b \in \Cols(\ECtwo)$
and $\wt{z} \in B_{b}$ imply that $d_{G_b}(\wt{z},A_b) \geq (\frac{1}{2}-2\eps)n$
and $|V(D') \cap A_b| \geq |V(D')| - |A \sd A_b| \geq 10\eta n - 2\eta n -\dD n \geq 7\eta n$
which is larger than the $2\eta n$ vertices used so far.
Greedily construct a path of unused vertices inside $D'$ starting at $z$, ending at some $w \in B_{c'} \cap B_1$ where $c' \in \Cols(\ECtwo)$ is unused, and consisting of $t$ vertices.
Greedily assign unused colours from the lists guaranteed by the definition of $D$.
The path $P$ obtained by concatenating this transversal path with $P_1$
has endpoints $\wt{y},w$ and has all of the required properties.
\end{proof}

The final lemma of this section combines the previous ones to find a transversal Hamilton cycle
in three cases.
The proof proceeds by using Lemma~\ref{lm:cover1} or~\ref{lm:cover2} to find a short transversal path covering bad vertices and colours, and then covering the remaining vertices with three long paths guaranteed by Lemma~\ref{lm:almostcp1}.

\begin{lemma}\label{lm:rhc}
Suppose that~\ref{it:dagger} holds,
along with one of the following:
\begin{align*}
\text{either } (i)\quad &|\Cols(\ECone)| < \eta n
\quad\text{and}\quad
\min\left\{ e_{\bm{G}(\ECtwo)}(A_1),
e_{\bm{G}(\ECtwo)}(B_1) \right\}
\geq 30\eta n^3;\\
\text{or }(ii)\quad &|\Cols(\ECone)| \geq \eta n
\quad\text{and}\quad
\max\left\{ e_{\bm{G}(\ECtwo)}(A_1),
e_{\bm{G}(\ECtwo)}(B_1)\right\}
\geq 30\eta n^3;\\
\text{or }(iii)\quad 
&|\Cols(\ECone)| \geq \eta n
\quad\text{and}\quad
|\Cols(\ECtwo)| \geq \eta n
\quad\text{and}\quad G_n \cong K_n.
\end{align*}
Then $\bm{G}$ contains a transversal Hamilton cycle.
\end{lemma}

\begin{proof}
Suppose that (i) holds.
By Lemma \ref{lm:cover2}, 
there is a transversal path $P$ with endpoints $x \in A_c$ and $y \in B_{c'}$
where $c,c' \in \Cols(\ECtwo) \sm \col(P)$ are distinct,
$\ms{A} := [n] \sm (\col(P)\cup\{c,c'\}) \subseteq \Cols(\ECtwo)$
and $|A|=|B| \geq (\frac{1}{2}-4\eta)n$
where $Z := Z_1 \sm V(P)$ for $Z \in \{A,B\}$.
So
$$
|\ms{A}|=n-|\col(P)|-2=|A|+|B|+|V(P)|-|\col(P)|-2=|A|+|B|-1.
$$
Let $A^- := N_{G_{c'}}(y,A)$ and $B^+ := N_{G_c}(x,B)$.
By Lemma~\ref{lm:almostcp1} there is a transversal path $P'$
with colour set $\ms{A}$
starting at some $x' \in A^-$ and ending at some $y' \in B^+$.
Concatenating this with $P$ using the connecting edges $yx'$ of colour $c'$
and $y'x$ of colour $c$, we obtain a transversal Hamilton cycle, proving Case~(i). 

Thus it remains to consider Cases~(ii) and~(iii).
In Case~(ii), set $\ms{F} := \emptyset$ and in Case~(iii), set $\ms{F} := \{n\}$.
In both cases, we have $|\Cols(\ECtwo)| \geq \eta n$.
Thus we can apply Lemma~\ref{lm:cover1} with $\ms{F}$ to obtain a transversal path $P$ with endpoints $x,y\notin V_{\rm bad}$ such that 
$V_{\rm bad} \subseteq V(P)$ and $\Cols_{\rm bad} \sm \ms{F} \subseteq \col(P)$
and $|V(P)| \leq 19\sqrt{\dD}n$, and $\ms{F} \cap (\col(P) \cup \{c,c'\} = \emptyset$, and
$x\in A_1 \cap A_c$ and $y\in B_1 \cap B_{c'}$ for some colours $c,c'\in \Cols(\ECtwo)$. 
Let $Z := Z_1 \sm V(P)$ for $Z \in \{A,B\}$.

In the next step, we proceed differently in each case.
In Case~(ii), by symmetry, we assume that $e_{G(\ECtwo)}(A_1)\geq 30\eta n^3$.
Let $D:=\{uv : u,v\in A\text{ and }c(uv)\geq 10\eta n\}$ as defined in the proof of  Lemma \ref{lm:cover2} 
where $c(uv):=\{i\in \Cols(\ECtwo) : uv\in E(G_i)\}$. 
There is a subgraph $D'$ of $D$ with minimum degree at least $10\eta n$.
In Case~(iii), let $D'$ be an arbitrary clique of size $10\eta n$ inside $A$.

We resume a unified approach for both cases.
Let $c_1,c_2 \in \Cols(\ECtwo)$ be distinct unused colours. 
Let $\ms{A}(H) := \Cols(H) \sm (\col(P) \cup \{c,c',c_1,c_2\})$
for $H \in \{\ECone, \ECtwo\}$.
We have
$$
|A|+|B|-3=n-|V(P)|-3 = n-|\col(P)|-4 = |\ms{A}(\ECone)|+|\ms{A}(\ECtwo)|
$$
and $\left| |A|-|B|\right| \leq 5\dD n$.
\begin{figure}
\begin{center}
\includegraphics{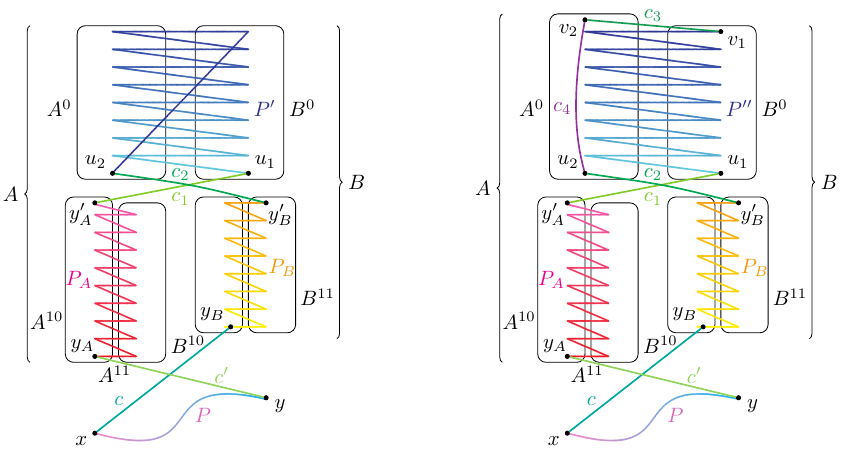}
\end{center}
\caption{Finding a transversal Hamilton cycle in Lemma~\ref{lm:rhc}(ii) and~(iii). Here, $|A^{10}|-|A^{11}|=1$ and $|B^{10}|-|B^{11}|=0$. On the left, $|A^0|-|B^0|=0$,
while on the right, $|A^0|-|B^0|=1$.}\label{fig:ext}
\end{figure}
Choose partitions
$A = A^0 \cup A^1$ and $A^1 = A^{10} \cup A^{11}$,
and
$B = B^0 \cup B^1$ and $B^1 = B^{10} \cup B^{11}$
such that
$$
0 \leq |A^{0}| - |B^{0}| \leq 1,
\quad
|\ms{A}(\ECtwo)|=|A^0|+|B^0|-1
\quad\text{and}\quad
\dD(D'[V(D') \cap A^{0}]) \geq \eta^2 n,
$$
and a partition $\ms{A}(\ECone) = \Cols_A \cup \Cols_B$ 
with
$$
0 \leq |A^{10}|-|A^{11}|,|B^{10}|-|B^{11}| \leq 1,
\quad
|\Cols_A|=|A^{10}|+|A^{11}|-1
\quad\text{and}\quad
|\Cols_B|=|B^{10}|+|B^{11}|-1.
$$
The only non-trivial part of this is the assertion about $D'$
which holds since $|A^0| \geq |\Cols(\ECtwo)|/3 \geq \eta n/3$,
and we could take~e.g.~a random partition and appeal to a Chernoff bound.
Let 
$$
\bm{G}^0 =(G_i[A^0,B^0]:i \in \ms{A}(\ECtwo))\quad\text{and}\quad
\bm{G}^Z =(G_i[Z^{10},Z^{11}]:i \in \Cols_Z)\quad\text{ for }Z \in \{A,B\}.
$$
Lemma~\ref{lm:almostcp1} with $W=Z$ implies that there is a transversal path $P_A$ inside $\bm{G}^A$ with colour set $\Cols_A$ and transversal path $P_B$ inside $\bm{G}^B$ with colour set $\Cols_B$ such that $P_A$ starts at some $y_A\in N_{G_{c'}}(y) \cap A^{10}$ and ends at some $y_A'\in A_{c_1} \cap A^1$, and $P_B$ starts at some $y_B\in N_{G_{c}}(x) \cap B^{10}$ and ends at some $y_B'\in B_{c_2} \cap B^1$.
(For example, to find $P_A$, we take $(W^*,Z^*)=(A^{10},A^{11})$.)
To complete the proof, it suffices to find a transversal path $P'$ in $\bm{G}^0$ with colour set $\ms{A}(\ECtwo)$,
with startpoint $u_1 \in B^0$ where $c_1 \in L(u_1y_A')$, and endpoint $u_2 \in A^0$ where $c_2 \in L(u_2y_B')$.
Then $xPyy_AP_Ay_A'u_1P'u_2y_B'P_By_Bx$ 
with colours $\col(P),c',\col(P_A),c_1,\col(P'),c_2,\col(P_B),c$
is a transversal Hamilton cycle in $\bm{G}$.

Suppose first that $|\ms{A}(\ECtwo)|$ is odd. 
Then $|A^0|=|B^0|$ and Lemma~\ref{lm:almostcp1} with $(W^*,Z^*)=(B^0,A^0)$, $t=0$, $W^-=N_{G_{c_1}}(y_A') \cap B^0$ and $T^+=N_{G_{c_2}}(y_B') \cap A^0$ immediately implies the existence of the required path between $u_1 \in W^-$ and $u_2 \in T^+$
(see the left-hand side of Figure~\ref{fig:ext}).

Suppose secondly that $|\ms{A}(\ECtwo)|$ is even.
So $|A^0|=|B^0|+1$.
Choose a colour $c_3 \in \ms{A}(\ECtwo)$ and 
$v_2\in A_{c_3} \cap A^0$ which is a vertex of $D'$.
Then choose $u_2 \in N_{G_{c_2}}(y_B') \cap A^0$ which is a neighbour of $v_2$ in $D'$.
These choices are both possible since there are at least $\eta^2n$ such neighbours in $D'$
of which at most $2\dD n$ are forbidden due to the $G_{c_2}$ neighbourhood condition.
In Case~(ii), assign an unused colour $c_4\in \ms{A}(\ECtwo)$ to $u_2v_2$ using the large guaranteed colour lists. 
In Case~(iii), assign the colour $c_4 :=n$ to $u_2v_2$.
In both cases, $u_2v_2 \in G_{c_4}$.

Now let ${\bm{G}^0}'=(G_i[A^0 \sm \{u_2,v_2\},B^0]: i\in \ms{A}(\ECtwo)\sm \{c_3, c_4\})$.
By Lemma \ref{lm:almostcp1} with $(W^*,Z^*)=(B^0,A^0 \sm \{u_2,v_2\})$ and $t=1$,
there is a transversal path $P''$ inside ${\bm{G}^0}'$ with startpoint $u_1\in N_{G_{c_1}}(y_A') \cap B^0$ and endpoint $v_1\in N_{G_{c_3}}(v_2) \cap B^0$ using colour set $\ms{A}(\ECtwo) \sm \{c_3,c_4\}$. 
We set 
$P' := u_1P''v_1v_2u_2$, using $\col(P''),c_3,c_4$ in that order
(see the right-hand side of Figure~\ref{fig:ext}).
This completes the proof.
\end{proof}

\section{Proofs of Theorems~\ref{th:rhc} and~\ref{th:rhp}}\label{sec:proof}

Finally, we can put all the ingredients together to prove our two main theorems on stability for transversal Hamilton cycles and paths.

\begin{proof}[Proof of Theorems \ref{th:rhc} and~\ref{th:rhp}]
Let $\kK>0$ and assume that $\kK<1$ or the theorems are both vacuous.
Choose additional parameters $n,\mu,\aA,\gG,\eps,\dD,\eta$ such that $n \in \mb{N}$ and
$$
0<1/n\ll \mu,\aA \ll \gG,\eps\ll \dD \ll \eta \ll \kK<1
$$
such that the conclusions of Lemma~\ref{lm:char} and Lemmas~\ref{lm:rhcstable}--\ref{lm:rhc} hold.
Let $n' \in \{n-1,n\}$ and 
let $\bm{G} =(G_1,\ldots,G_{n'})$
be a graph collection on a common vertex set $V$ of size $n$
with $\dD(\bm{G}) \geq (\frac{1}{2}-\mu)n$.

If $n'=n-1$, let $G_n := K_n$ and $\bm{J} :=(G_1,\ldots,G_n)$.
If $n'=n$, let $\bm{J} := \bm{G}$.
Suppose that $\bm{J}$ does not contain a transversal Hamilton cycle.
Then Lemma \ref{lm:rhcstable} implies that $\bm{J}$ is not $(\gG, \aA, \eps, \dD)$-stable.
Thus, without loss of generality, for every colour $i\in [(1-\gG)n]$,
the graph $G_i$ is not $\aA$-nice and hence is $\eps$-extremal
with characteristic partition $(A_i,B_i,C_i)$,
and $e(C^{\eps,\dD}_{\bm{J}}) < \dD n^2$.
We may further assume that $1$ is the colour of minimum degree in the cross graph $C^{\eps,\dD}_{\bm{J}}$, and $[(1-\gG)n] \sm N_{C_{\bm{J}}^{\eps,\dD}}(1) = [m]$. 
Let $\Cols_{\rm bad}:=[m+1,n]$ be the set of excluded colours and thus $|\Cols_{\rm bad}|\leq \gG n+2\dD n \leq 3\dD n$. 
It follows from Lemma \ref{lm:char}~that $e_{\bm{J}(\ECone)}(A_1,B_1)\leq \eps n^3$.
For every $i \in [m]$, $G_i$ and $G_1$ are not $\dD$-crossing
and hence we can swap the labels of $A_i$ and $B_i$ if necessary to get that
$|A_1 \sd A_i|,|B_1 \sd B_i| \leq \dD n$.
That is, $\bm{J}$ satisfies~\ref{it:dagger}.
We use the same notation defined after~\ref{it:dagger}.

Suppose that one of the following hold.
\begin{itemize}
\item[(i)] $|\Cols(\ECone)|\leq \eta n$ and either $e_{\bm{J}(\ECtwo)}(A_1)\leq 30\eta n^3$ or $e_{\bm{J}(\ECtwo)}(B_1)\leq 30\eta n^3$; or
\item[(ii)] $|\Cols(\ECone)|\geq \eta n$ and $e_{\bm{J}(\ECtwo)}(A_1)+e_{\bm{J}(\ECtwo)}(B_1)\leq 60\eta n^3$; or
\item[(iii)] $G_n=K_n$ and $|\Cols(\ECtwo)| < \eta n \leq |\Cols(\ECone)|$.
\end{itemize}
We claim that in these cases, $\bm{J}$ is $\kK$-close to, respectively,
\begin{itemize}
\item[(i)] 
a half-split graph collection;
\item[(ii)] 
$\bm{H}^b_a$ where $a=|\Cols(\ECone)|$ and $b=|\Cols(\ECtwo)|\pm 1$ is odd;
\item[(iii)] 
$\bm{H}^0_n$.
\end{itemize}
Indeed, for~(i), we can remove at most $30\eta n^3$ edges so that some $Z_1 \in \{A_1,B_1\}$ has no $\bm{J}(\ECtwo)$-edges, and delete all edges in graphs in $\bm{J}(\ECone) \cup \Cols_{\rm bad}$ by removing at most $\eta n^3+3\dD n^3$ edges.
Finally, edit at most $2\dD n^3$ edges to increase $|Z_1|$ by less than $\eps n$ so it has size $\lfloor n/2\rfloor+1$ and make $(Z_1,\ov{Z_1})$ complete in every graph in $\bm{J}(\ECtwo)$.
The resulting graph collection is half-split and in total we have made $31\eta n^3+5\dD n^3 < \kK n^3$ edits.
For~(ii), we can edit at most $3\dD n^3+\eps n^3+60\eta n^3$ edges
so that for every $i \in \Cols_{\rm bad}$ the graph $G_i$ becomes a copy of $\ECone$ whose parts contain $A_1,B_1$; 
$e_{G_i}(A_1,B_1)=0$ for all $G_i \in \bm{J}(\ECone)$;
and $e_{\bm{J}(\ECtwo)}(A_1)+e_{\bm{J}(\ECtwo)}(B_1)= 0$.
A further at most $\dD n^3$ edits will make $\bm{J}$ isomorphic to $\bm{H}_a^b$
where $a=|\Cols(\ECone)|$ and $b=|\Cols(\ECtwo)| \pm 1$ is odd.
Thus we make at most $61\eta n^3 < \kK \eta n^3$ edits in total.
For~(iii), we can make at most $2\eta n^3+2\dD n^3 < \kK n^3$ edits to make $\bm{J}$ isomorphic to $\bm{H}_n^0$.
This proves the claim.

Now we prove Theorem~\ref{th:rhc}. Here, $\bm{J}=\bm{G}$.
Lemma~\ref{lm:rhc} implies that, if $|\Cols(\ECone)|<\eta n$, then~(i) holds,
while if $|\Cols(\ECone)| \geq \eta n$, then~(ii) holds.
This completes the proof of the theorem.

Finally, we prove Theorem~\ref{th:rhp}.
Here, $n'=n-1$ and $\bm{J}=\bm{G} \cup \{K_n\}$, so $\bm{G}$ does not have a transversal Hamilton path
if and only if $\bm{J}$ does not have a transversal Hamilton cycle.
Lemma~\ref{lm:rhc} implies that, if $|\Cols(\ECone)| < \eta n$, then~(i) holds,
while if $|\Cols(\ECone)| \geq \eta n$, then~(iii) holds 
(and~(ii) also holds in this case).
This completes the proof of the theorem.
\end{proof}

\section{Concluding remarks}\label{sec:conclude}

In this paper, we proved Theorem~\ref{th:rhc} that any collection of $n$ almost-Dirac graphs on the same large set of $n$ vertices either has a transversal Hamilton cycle, or is close to one of several types of collection
which do not contain a transversal Hamilton cycle.
We proved Theorem~\ref{th:rhp}, an analogue for Hamilton paths, 
characterising collections of $n-1$ almost-Dirac graphs without transversal Hamilton paths.

We proved these theorems in a unified manner, combining the regularity-blow-up method and absorption method for transversals.
We believe that our method can be utilised to characterise stability for other spanning transversal embedding problems.
It provides some hope for proving exact results: i.e.~determining the best possible transversal minimum degree threshold.
Indeed, while the transversal minimum degree thresholds for Hamilton cycles and perfect matchings are known by the results of Joos and Kim~\cite{JK},
exact results are commonly proved by first proving stability by showing that any graph without the required subgraph $H$ and almost the conjectured minimum degree must have a specific structure.
Then, an extremal analysis is conducted to show that there cannot be any imperfections in this structure.
Given the literature on classical embedding, it seems unlikely that a short `elementary' argument (i.e.~without using any machinery) of the type used by Joos and Kim can be used to find exact thresholds for many of the natural graphs $H$ in which we are interested.
In particular, we think the ideas developed here may be useful in 
resolving the following conjecture,
a transversal analogue of the Hanjal-Szemer\'edi theorem~\cite{HS},
which would indeed generalise this theorem.

\begin{conj}[Transversal Hajnal-Szemer\'edi \cite{Cheng2}]\label{cj:kr}
Let $k \geq 2$ be an integer and let $n$ be a sufficiently large multiple of $k$.
Let $\bm{G}=(G_1,G_2,\ldots,G_{\frac{n}{k}\binom{k}{2}})$ be a graph collection on a common vertex set of size $n$.
Suppose $\dD(\bm{G})\geq (1-\frac{1}{k})n$.
Then $\bm{G}$ contains a transversal copy of a $K_k$-factor.
\end{conj}
As mentioned, the case $k=2$ of perfect matchings was proved by Joos and Kim in~\cite{JK},
and the asymptotic version of this conjecture was proved independently in \cite{Cheng2} and \cite{MMP}.

\bibliographystyle{plain}
\bibliography{Bibte}

\end{document}